\documentclass{cocv}

\usepackage{xcolor}
\usepackage{latexsym}
\usepackage{amsmath}
\usepackage{amsfonts}
\usepackage{amssymb}
\usepackage{amsthm}
\usepackage{amssymb,epsfig}
\usepackage{t1enc}
\usepackage[english]{babel}

\usepackage{hyperref}
\hypersetup{bookmarksopenlevel=2}
\makeatletter
\DeclareOldFontCommand{\rm}{\normalfont\rmfamily}{\mathrm}
\DeclareOldFontCommand{\sf}{\normalfont\sffamily}{\mathsf}
\DeclareOldFontCommand{\tt}{\normalfont\ttfamily}{\mathtt}
\DeclareOldFontCommand{\bf}{\normalfont\bfseries}{\mathbf}
\DeclareOldFontCommand{\it}{\normalfont\itshape}{\mathit}
\DeclareOldFontCommand{\sl}{\normalfont\slshape}{\@nomath\sl}
\DeclareOldFontCommand{\sc}{\normalfont\scshape}{\@nomath\sc}
\makeatother

\usepackage{csquotes}  
\newcommand\q{\enquote}

\usepackage{physics}   
\usepackage[normalem]{ulem} 
\usepackage{cancel}

\newcommand{\myV}{\tilde{V}}

\newcommand \N   {\mathbb{N}}
\newcommand \R   {\mathbb{R}}

\newcommand{\Uc}{\ensuremath{\mathcal{U}}}

\newcommand{\vertiii}[1]{{\left\vert\kern-0.25ex\left\vert\kern-0.25ex\left\vert #1 
    \right\vert\kern-0.25ex\right\vert\kern-0.25ex\right\vert}}

\newcommand \qrq   {\quad\Rightarrow\quad}

\newcommand{\normt}[1]{{\left\vert\kern-0.25ex\left\vert\kern-0.25ex\left\vert #1 
		\right\vert\kern-0.25ex\right\vert\kern-0.25ex\right\vert}}

\newcommand{\intt}{{\rm int}\,}

\newif\ifMath					
\newif\ifEngi					

\newif\ifDFGtext					 

\newif\ifAndo              
													
\newif\ifExercises					
\newif\ifSolutions          
\newif\ifGerman							
\newif\ifEnglish						

\newif\ifnothabil						

\newif\ifFuture							

\newif\ifConf                    
\newif\ifJournal								 

\newif\ifNOTFORBOOK
\newif\ifFullVersion
\newif\ifExludedDueToSpaceReasons

\usepackage{xifthen}

\newcommand{\einsnorm}[2]{\ensuremath{
    \!\!\;\!\!\!\;
    \left\bracevert\!\!\!\!\!\left\bracevert
    \!
		\ifthenelse{\isempty{#2}}{#1}{#1(#2)}
    \!
      \right\bracevert\!\!\!\!\!\right\bracevert
    \!\!\;\!\!\!\;
  }}

\usepackage{xcolor}
\definecolor{blond}{rgb}{0.98, 0.94, 0.75}
	
\newlength\mytemplen
\newsavebox\mytempbox

\makeatletter
\newcommand\mybluebox{%
    \@ifnextchar[
       {\@mybluebox}%
       {\@mybluebox[0pt]}}

\def\@mybluebox[#1]{%
    \@ifnextchar[
       {\@@mybluebox[#1]}%
       {\@@mybluebox[#1][0pt]}}

\def\@@mybluebox[#1][#2]#3{
    \sbox\mytempbox{#3}%
    \mytemplen\ht\mytempbox
    \advance\mytemplen #1\relax
    \ht\mytempbox\mytemplen
    \mytemplen\dp\mytempbox
    \advance\mytemplen #2\relax
    \dp\mytempbox\mytemplen
    \colorbox{blond}{\hspace{1em}\usebox{\mytempbox}\hspace{1em}}}

\makeatother

\makeatletter
\let\origd=\d
\renewcommand*\d{
  \relax\ifmmode
    \mathrm{d}%
  \else
    \expandafter\origd
  \fi
}\makeatother

\usepackage{mathtools}

\makeatletter 
\newcommand{\pushright}[1]{\ifmeasuring@#1\else\omit\hfill$\displaystyle#1$\fi\ignorespaces}
\newcommand{\pushleft}[1]{\ifmeasuring@#1\else\omit$\displaystyle#1$\hfill\fi\ignorespaces}
\makeatother 

\newcounter{syscounter}

\newcounter{WPcounter}
\newcounter{PRcounter}
\def \d {\displaystyle}

\newtheorem{thm}{Theorem}[section]
\newtheorem{prop}[thm]{Proposition}
\newtheorem{lem}[thm]{Lemma}
\newtheorem{cor}[thm]{Corollary}
\newtheorem{exm}[thm]{Example}
\newtheorem{dfn}[thm]{Definition}

\newtheorem{rem}[thm]{Remark}

\begin{document}
\title{Lyapunov methods for input-to-state stability of time-varying evolution equations}
\thanks{A. Mironchenko has been supported by the Heisenberg program of the German Research Foundation (DFG), grant MI1886/3-1.}
\thanks{R. Heni is a corresponding author.}

\author{Rahma Heni}\address{Faculty of Sciences of Sfax, University of Sfax, 3000 Sfax, Tunisia; \email{henirahmahenirahma@gmail.com \ mohamedali.hammami@fss.rnu.tn}}
\author{Andrii Mironchenko}\address{Department of Mathematics, University of Bayreuth, 95447 Bayreuth, Germany; \email{andrii.mironchenko@uni-bayreuth.de}}
\author{Fabian Wirth}\address{Faculty of Computer Science and Mathematics, University of Passau, 94032 Passau, Germany; \email{fabian.wirth@uni-passau.de}}
\author{Hanen Damak}\address{Faculty of Economics and Management of Sfax, University of Sfax, 3018 Sfax, Tunisia; \email{hanen.damak@yahoo.fr}}
\author{Mohamed Ali Hammami}\sameaddress{1}
%
%
\begin{abstract} We prove that (local) input-to-state stability ((L)ISS) and integral input-to-state stability (iISS) of time-varying infinite-dimensional systems in abstract spaces follows from the existence of a {corresponding} Lyapunov function. 
In particular, input-to-state stability of linear time-varying control
systems in Hilbert spaces with bounded input operators is discussed. Methods for the construction of non-coercive LISS/iISS Lyapunov functions are presented for a certain class of time-varying semi-linear evolution equations. Two examples are given to illustrate the effectiveness of the results. 
\end{abstract}
%
%
\maketitle
\section*{Introduction}
The concept of input-to-state stability (ISS) proposed in the late 1980s by E. Sontag \cite{Son89} is one of the central notions in robust nonlinear control. ISS has become indispensable for various branches of nonlinear systems theory, such as robust stabilization of nonlinear systems \cite{FrK96}, design of nonlinear observers \cite{ArK01}, analysis of large-scale networks \cite{DRW07,JTP94}, etc.
To study the robustness of systems with saturation and limitations in actuation and processing rate (which are typically not ISS), a notion of integral input-to-state stability (iISS) has been proposed in \cite{sontag1998comments}.
The equivalence of the ISS and the existence of an ISS Lyapunov function in a dissipative form was proved for time-invariant ordinary differential equations (ODEs) with Lipschitz right-hand side in \cite{SoW95}. An analogous Lyapunov characterization was shown for iISS in \cite{ASW00}. An important implication of these results is that for smooth enough systems, ISS implies iISS.
For the overview of the ISS theory of ODEs, we refer to the classic survey \cite{Son08} as well as a recent monograph \cite{Mir23}.

Notable efforts have been undertaken to extend the ISS concept to time-varying systems governed by ODEs in the past years, see \cite{TsK99,ELW00,KaT04,LWC05}.
In particular, in \cite{ELW00} the Lyapunov characterizations of ISS for time-varying nonlinear systems including
periodic time-varying systems have been studied.
As demonstrated in  \cite[p. 3502]{ELW00}, in contrast to time-invariant
ODE systems, in the time-variant case, ISS does not necessarily imply
iISS, even for systems with a sufficiently smooth right-hand side. A
detailed analysis of this question has been undertaken in \cite{HaM19}.

The success of the ISS theory of ODEs and the need for robust stability analysis of partial differential equations (PDEs) motivated the development of ISS theory in the infinite-dimensional setting, see \cite{DaM13,GGL21,jacob2018infinite,JSZ19,karafyllis2016iss,karafyllis2017iss,MiI15b,MiW18b,prieur2012iss,ZhZ18}.
For instance, the Lyapunov method was used in \cite{MiI15b} for analysis of iISS of nonlinear parabolic equations. In \cite{JMP20,MiW18b}, the notion of a non-coercive ISS Lyapunov function has been introduced and it was
demonstrated that the existence of such a function implies ISS for a broad class of nonlinear infinite-dimensional
systems.
In \cite{jacob2018infinite}, the relation between ISS and iISS has been studied for linear infinite-dimensional systems with a possibly unbounded control operator and inputs in general function spaces. In \cite{karafyllis2016iss}, the ISS property for 1-D parabolic PDEs with boundary disturbances has been considered. In \cite{ZhZ18}, ISS of a class of semi-linear parabolic PDEs with respect to boundary disturbances has been studied for the first time based on the Lyapunov method.
For an overview of the ISS theory for distributed parameter systems, we refer to \cite{CKP23,KaK19,MiP20}.

\textbf{Time-varying infinite-dimensional systems.}
Much less attention has been devoted to the ISS
of time-varying infinite-dimensional systems
\cite{damak2021input,hanen2022input}, a similar statement holds for iISS, \cite{mancilla2023characterization}.
For instance, in \cite{damak2021input}, a direct Lyapunov theorem was shown for a class of time-varying semi-linear evolution equations in Banach spaces with Lipschitz continuous nonlinearities.

The analysis of ISS for time-varying systems is more involved than that of time-invariant systems, even for linear systems. Consider an abstract Cauchy problem of the form

$$\dot{x}(t)=A(t)x(t),\quad x(t_0)=x_0, \quad t\geq t_0. $$

If $A$ does not depend on time, the Hille-Yosida theorem \cite[Theorem 3.1]{Paz83} gives a characterization of the infinitesimal
generators of strongly continuous semigroups, and hence, it characterizes all
well-posed autonomous abstract Cauchy problems.
At the same time, there is no extension of the Hille-Yosida theorem to the time-variant case.
The conditions that ensure that $\{A(t)\}_{t\geq0}$ generates an evolution family are well-understood if $\{A(t)\}_{t\geq0}$ is a family of bounded operators. However, when the operators $A(t)$ are unbounded, it is a delicate matter to prove the well-posedness of an abstract Cauchy problem. We refer the reader to
\cite{carmen,Paz83} and the references therein for more information. There are almost no results on ISS and iISS of time-varying systems with unbounded $A(t)$.

\textbf{Contribution.} Motivated by the foregoing discussion, this paper
investigates the (L)ISS/iISS property of abstract time-varying infinite-dimensional
systems in terms of (L)ISS/iISS Lyapunov functions. We also develop conditions
for the well-posedness of time-varying semi-linear evolution equations in
Banach spaces and derive stability results for evolution operators in
Banach spaces. Furthermore, it is shown that for each input-to-state
stable linear time-varying control system with a bounded linear part,
there is a coercive Lyapunov function. Also, we develop a method for the
construction of a non-coercive ISS Lyapunov function for a class of linear
time-varying control systems where the operators in the family
$\{A(t)\}_{t\geq0}$ are unbounded. Next, we derive a novel construction of
non-coercive LISS/iISS Lyapunov functions for a certain class of
time-varying semi-linear systems with the associated nominal system being
linear and unbounded that depends on the time.

The remainder of the paper is organized as follows. In
Section~\ref{sec:time-varying-control} an abstract class of time-varying
linear systems is introduced. In Section~\ref{sec:iss-lyap-funct}, the
central notion of input-to-state stability is defined in the abstract
context and it is shown how to use Lyapunov functions to verify the
(L)ISS/iISS of time-varying infinite-dimensional systems. We turn to more
concrete system descriptions in Section~\ref{sec:time-varying-semi}, where
a well-posedness analysis of time-varying semi-linear
evolution equations with locally Lipschitz continuous nonlinearities and
piecewise-right continuous inputs is presented. In Section~\ref{sec:lyap-crit-iss}, a Lyapunov
characterization for ISS of linear time-varying systems is presented in
Hilbert spaces. In Section 5, we show, how a non-coercive LISS/iISS
Lyapunov function can be constructed for time-varying semi-linear
systems. In Section~\ref{sec:examples}, the iISS/ISS analysis of
two parabolic PDEs is given to illustrate the proposed
method. Conclusions are drawn in Section~\ref{sec:conclusion}.

\textbf{Notation.} 
Throughout this paper, we adopt the following notation: $\mathbb{R}_{+}$
denotes the set of non-negative real numbers. By $\mathcal{T}$ we denote
the set of all the pairs $(t,s)\in \mathbb{R}_+^{2}$ with $t\geq s$. 
In a normed space $(X,\|\cdot\|_{X})$ the closed ball around zero is
denoted $B_{r}:=\{x\in X: \|x\|_{X}\leq r\}.$ The interior of 
$D\subset X$ is denoted by $\intt(D).$
For linear normed spaces
$X,Y$ let $L(X,Y)$ be the space of bounded linear operators from $X$ to
$Y,$ equipped with the usual operator norm $\|\cdot\|$; also $L(X) :=
L(X,X)$. If $X$ is a Hilbert space, then $A^{\ast}$ denotes the adjoint operator of
$A\in L(X)$. $I\in L(X)$ is the identity operator. By $C(X,Y)$, we denote the space of continuous functions from $X$ to $Y,$ $C(X):=C(X,X).$
$PC(\mathbb{R}_+, U)$ denotes the space of globally bounded, right
continuous, and piecewise continuous functions from $\mathbb{R}_+$ to $U$ with the norm $\|u\|_{\mathcal{U}}:=\sup_{0\leq s\leq \infty} \|u(s)\|_{U}.$
$PC^{1}(\mathbb{R}_+, U)$ denotes the space of 
continuous, piecewise right continuously differentiable functions from $\mathbb{R}_+$ to $U.$ We need the following spaces of integrable functions (here $\ell>0$):
\begin{itemize}
\item $L^{p}(0,\ell),\; 1\leq p < \infty$ is the space of $p$-th power integrable functions $f:(0,\ell)\to \mathbb{R}$ with the norm 
\[
\|f\|_{L^{p}(0,\ell)}=\left(\int _{0} ^{\ell}|f(x)|^{p}dx\right)^{\frac{1}{p}}.
\]
\item $C^{k}_{0}(0,\ell)$ is the space of $k$ times continuously
  differentiable functions $f:(0,\ell)\to \mathbb{R}$ with a compact support in $(0,\ell).$
\item $W^{p,k}(0,\ell)$ is a Sobolev space of functions $f\in L^{p}(0,\ell),$ which have weak derivatives of order $\leq k,$ all of which belong to $L^{p}(0,\ell).$
\item $W_{0}^{p,k}(0,\ell)$ is a closure of $C_{0}^{k}(0,\ell)$ in the norm of $W^{p,k}(0,\ell)$, $H^{k}(0,\ell)=W^{2,k}(0,\ell)$ and $H_{0}^{k}(0,\ell)=W_{0}^{2,k}(0,\ell).$
\end{itemize}
We use the following classes of comparison functions:
\begin{itemize}
\item $\mathcal{P}=\{\gamma:\mathbb{R}_+\to \mathbb{R}_+:\gamma \;\;\hbox{is continuous,} \gamma(0)=0\; \hbox{and}\;\gamma(r)>0 \;\hbox{for} \;r>0 \}.$
\item $\mathcal{K}=\{\gamma\in \mathcal{P}:\gamma \;\;\hbox{is strictly increasing}\}.$
\item $\mathcal{K}_{\infty}=\{\gamma\in \mathcal{K}:\gamma \;\;\hbox{is unbounded}\}.$
\item $\mathcal{L}=\{\gamma:\mathbb{R}_+\to \mathbb{R}_+:\gamma \;\;\hbox{is continuous and strictly decreasing with} \displaystyle\lim_{t\to \infty}\gamma(t)=0\}.$
\item $\mathcal{KL}=\{\beta\in C(\mathbb{R}_+\times\mathbb{R}_+,\mathbb{R}_+):\beta(\cdot,t)\in \mathcal{K}\ \forall t\geq0,\ \beta(r,\cdot)\in \mathcal{L}\ \forall r> 0\}.$
\end{itemize}
Functions in ${\cal P}$ are also called \emph{positive definite}.

\section{Time-varying control systems}
\label{sec:time-varying-control}

In this paper, we consider abstract axiomatically defined time-varying
systems on the state space $X$ following \cite{hinrichsen2005mathematical,hinrichsen2025mathematical,KaJ11,MiW18b,mancilla2023characterization}.
\begin{dfn}\label{csyol}(\hspace*{-3pt}\cite{mancilla2023characterization})
Consider a triple $\Sigma=(X,\mathcal{U},\phi),$ consisting of the following
components:
\begin{enumerate}
\item[$(i)$] A normed linear space $X,$ called \emph{the state space}, endowed with the
norm $\|\cdot\|_{X}$.
\item[$(ii)$] A set of admissible input values $U,$ which is a nonempty
  subset of a certain normed linear space with norm $\|\cdot\|_{U}$.
\item[$(iii)$] A normed linear space of input functions $\mathcal{U}
  \subset \{f: \mathbb{R}_+ \to U\}$ endowed with a norm
  $\|\cdot\|_{\mathcal{U}}.$ We assume that for ${\cal U}$ the following axiom holds:

\emph{Axiom of shift invariance}: for all $u\in\mathcal{U} $ and all $\tau\geq0,$ the time shift $u(\cdot+\tau)$ belongs to $\mathcal{U}$ with $\|u\|_{\mathcal{U}}\geq\|u(\cdot+\tau)\|_{\mathcal{U}}.$

\item[$(iv)$] A set $D_{\phi}\subseteq \mathcal{T}\times X\times
  \mathcal{U}$, and map $\phi:D_{\phi}\to X$, called the \emph{transition map}.
   For all $(t_0,x_0,u) \in \mathbb{R}_+\times X \times
  \mathcal{U},$ there exists a
  $t_m=t_m(t_0,x_0,u) \in (t_0,\infty) \cup\{\infty\}$ with
 $D_{\phi} \cap \bigl( [t_0,\infty)\times (t_0,x_0, u)\bigr)= [t_0,t_m)\times
 (t_0,x_0, u)$. 
\par The interval $[t_0,t_m)$ is called \emph{the maximal interval of
  existence of $t \mapsto \phi(t,t_0,x_0,u)$}, i.e., of the trajectory corresponding to the initial condition $x(t_0)=x_0$ and the input $u.$
\end{enumerate}
The triple $\Sigma$ is called \emph{a (control) system} if the following properties hold:
\begin{enumerate}
\item[$(\Sigma_{1})$] \emph{Identity property}: for every $(t_0,x_0,u)\in \mathbb{R}_+\times X\times \mathcal{U}$ it holds that $\phi(t_0,t_0,x_0,u)=x_0.$

\item[$(\Sigma_{2})$] \emph{Causality}: for every $(t, t_0, x_0,u)\in D_{\phi},$ and for every $\tilde{u}\in \mathcal{U}$ such that
$\tilde{u}(\tau)=u(\tau)$ for all $\tau\in [t_0,t),$ it holds that $(t,t_0,x_0,\tilde{u})\in D_{\phi}$ and $\phi(t,t_0,x_0,u)=\phi(t,t_0,x_0,\tilde{u}).$
\item[$(\Sigma_{3})$] \emph{Continuity}: for every $(t_0,x_0,u)\in
  \mathbb{R}_+\times X\times \mathcal{U}$ the map $t \mapsto
  \phi(t,t_0,x_0,u)$ is continuous on its maximal interval of existence $[t_0,t_m(t_0,x_0,u))$.
\item[$(\Sigma_{4})$] \emph{The cocycle property}: for all $t_0\geq0,$ all $x_0\in X,$ all $u\in \mathcal{U}$, and all $t\in [t_0,t_m(t_0,x_0,u)),$ we have
$\phi(t,\tau ,\phi(\tau,t_0,x_0,u),u)=\phi(t,t_0,x_0,u),$ for all $\tau\in [t_0,t].$
\end{enumerate}
\end{dfn}
\begin{dfn}
We say that a control system $\Sigma=(X,\mathcal{U},\phi)$ is \emph{forward complete} if $D_{\phi}=\mathcal{T}\times X\times\mathcal{U}$, that is, for every $(t_0,x_0,u) \in \mathbb{R}_+\times X\times \mathcal{U}$ and for all $t\geq t_0,$ the value $\phi(t,t_0,x_0,u)\in X$ is well-defined.
\end{dfn}

For a particular trajectory, $t\mapsto \phi(t,t_0,x_0,u)$, $t \in [t_0,
t_m(t_0,x,u))$, we will say that the trajectory \emph{exists globally} if
$t_m(t_0,x_0,u)= \infty$.
In general, the verification of forward completeness for time-varying
nonlinear systems is a complex task. In this context, a useful property that is satisfied by many time-varying infinite-dimensional control systems is the possibility to prolong bounded solutions to a larger interval.

\begin{dfn}(\cite{KaJ11})
We say that a control system $\Sigma=(X,\mathcal{U},\phi)$ satisfies the
\emph{boundedness-implies-continuation (BIC)} property if for each
$(t_0,x_0,u)\in \mathbb{R}_+\times X \times \mathcal{U},$ with
$t_m=t_m(t_0,x_0,u) < \infty$, and for every $M>0,$ there exists $t\in [t_0,t_m)$ with $\|\phi(t,t_0,x_0,u)\|_{X}> M.$
 \end{dfn}
To study the stability properties of time-varying control systems with
respect to external inputs, we introduce the following notion in the
spirit of \cite{Son89,MiW18b,mancilla2023characterization,TsK99}.
\begin{dfn}
System $\Sigma=(X,\mathcal{U},\phi)$ is called \emph{locally
  input-to-state stable} \emph{(LISS)} if there exist $\rho_{x},\rho_u>0$
and $\beta \in \mathcal{KL}, \gamma \in \mathcal{K},$ such that for all
$x_0$, $\|x_0\|_{X}\leq\rho_{x},$ all $t_0\geq 0$, and all $u \in
\mathcal{U}$, $\|u\|_{\mathcal{U}}\leq \rho_u$, the trajectory
$\phi(\cdot,t_0,x_0,u)$ exists on $[t_0,\infty)$ and for all $t\geq t_0$ it holds that
\begin{equation}\label{rhm21}
 \|\phi(t,t_0,x_0,u)\|_{X}\leq \beta(\|x_0\|_{X},t-t_0)+\gamma(\|u\|_{\mathcal{U}}).
\end{equation}
The control system is called \emph{input-to-state stable} \emph{(ISS)} if $\rho_x,\rho_u$ can be chosen to be equal to $\infty.$
\end{dfn}
 

For systems with piecewise continuous inputs, we define an ISS-like property in terms of the integral (\q{energy}) of the input.
\begin{dfn}\label{rhan23}
Let $\Sigma=(X, \mathcal{U},\phi)$ be a forward complete control system with $\mathcal{U}=PC(\mathbb{R}_+,U).$ System $\Sigma$ is called \emph{integral input-to-state stable (iISS)} if there exist $\alpha \in \mathcal{K}_{\infty},$ $\mu\in \mathcal{K}$ and $\beta\in\mathcal{KL}$, such that for all $x_0\in X,$ all $u\in \mathcal{U},$ all $t_0\geq 0,$ and all $t\geq t_0$ the following holds
\begin{equation}\label{EQua1in}
\|\phi(t,t_0,x_0,u)\|_{X}\leq \beta(\|x_0\|_{X},t-t_0)+\alpha \bigg{(}\int_{t_0} ^{t} \mu(\|u(s)\|_{U})ds\bigg{)}.
\end{equation}
\end{dfn}
\begin{dfn}
System $\Sigma=(X,\mathcal{U},\phi)$ is called \emph{uniformly globally
  asymptotically stable with zero input (\emph{0-UGAS})} if there exists $\beta \in \mathcal{KL}$
such that for all $x_0\in X$ and all $t_0\geq 0$ the corresponding trajectory $\phi(\cdot,t_0,x_0,0)$ exists globally, and for all $t\geq t_0$ it holds that
\begin{equation}\label{rhm1}
\|\phi(t,t_0,x_0,0)\|_{X}\leq \beta(\|x_0\|_{X},t-t_0).
\end{equation}
\end{dfn}
Substituting $u:=0$ into (\ref{rhm21}), we see that ISS systems are always 0-UGAS.

\section{ISS Lyapunov functions}
\label{sec:iss-lyap-funct}

In this section, we recall the concept of an ISS Lyapunov function. It is crucial for the verification of input-to-state stability of time-varying control systems.

Let $(t,x)\in \mathbb{R}_+ \times X$ and $W$ be an open neighborhood
  of $(t,x)$. Assume that $V:W \to \R_+$ is continuous. The Lie derivative of $V$ at $(t,x)\in \mathbb{R}_+ \times X$ corresponding to the input $u\in \mathcal{U}$ along the corresponding trajectory of $\Sigma=(X,\mathcal{U},\phi)$ is defined by
\begin{equation}
    \label{eq:Liederivativedef}
 \dot{V}_{u}(t,x):=D^{+}V(s,\phi(s,t,x,u))|_{s=t}:=\limsup\limits_{h\to0^+}\frac{1}{h}\Big(V(t+h,\phi(t+h,t,x,u))-V(t,x)\Big),
\end{equation}
where $D^{+}$ denotes the right upper Dini derivative. For this
  definition, note that $t_m(t,x,u) > t$ by assumption, so strictly
  speaking we define the limsup in \eqref{eq:Liederivativedef} by
  restricting $h>0$ to be sufficiently small.

\begin{dfn}\label{DEFINITION2}
Let $\Sigma=(X,\mathcal{U},\phi)$ be a time-varying control system, 
and $D\subset X$ be open with $0\in D$. A continuous function $V:\mathbb{R}_+\times D \to \mathbb{R}_+$ 
with $V(t,0)=0$, for all $t\geq 0$,  is called a \emph{non-coercive LISS Lyapunov function in implication form}
for the system $\Sigma$, if there exist $\alpha_2\in\mathcal{K}_{\infty}$,
$\kappa\in\mathcal{K}$,  a positive definite function $\mu$, and constants
$r_{1},r_{2}>0$, with $B_{r_1}\subset D,$ and such that
\begin{equation}\label{R1EN}
0<V(t,x)\leq \alpha_2(\|x\|_{X}) \quad \forall x\in D\backslash\{0\} \quad \forall t\geq0,
\end{equation}
and for all $t\geq0,$ all $x\in B_{r_{1}},$ all $u\in
\mathcal{U}$ with $\|u\|_{\mathcal{U}}\leq r_{2}$ it holds that
\begin{equation}\label{STAvNNv}
\|x\|_{X}\geq\kappa(\|u\|_{\mathcal{U}}) \quad\Longrightarrow\quad \dot{V}_{u}(t,x)\leq -\mu(V(t,x)).
\end{equation}
If, in addition, there is $\alpha_1\in\mathcal{K}_{\infty},$ so that
\begin{equation}\label{R1E}
 \alpha_1(\|x\|_{X}) \leq V(t,x)\leq \alpha_2(\|x\|_{X}) \quad \forall x\in D \quad \forall t\geq0,
\end{equation}
then $V$ is called a \emph{(coercive) LISS Lyapunov function in implication form}
for $\Sigma.$ The function $\kappa$ is called \emph{ISS Lyapunov gain} for $\Sigma.$

If in the previous definition $D=X,$ $r_1=\infty,$ and $r_2=\infty,$ then
$V$ is called a  non-coercive (coercive, if \eqref{R1E} holds) ISS Lyapunov function in implication form for $\Sigma.$
\end{dfn}

Since the flow $\phi$ is only continuous in time, the map $t\mapsto V(t,\phi(t,t_0,x_0,u))$ is only continuous, even if $V$ has a higher regularity.
We need the following comparison principle.
{
\begin{lem}\label{lemmm1}
For any $\theta\in \mathcal{P}$ there exists $\beta\in \mathcal{KL}$ so
that: If the differential inequality
\begin{equation}\label{bbb2}
D^{+}\omega(t)\leq-\theta(\omega(t)),\quad \forall\, t \geq t_0
\end{equation}
holds for a certain continuous function $\omega:\mathbb{R}_+\to \mathbb{R}_+$ and some $t_0\geq 0$, then it holds also
\begin{equation}
\label{eq:conclusionlemma1}
 \omega(t)\leq \beta(\omega(t_0),t-t_0)\quad 
\forall t\geq t_0.   
\end{equation}
\end{lem}
}
We skip the proof because it closely parallels the proof of \cite[Proposition A.35. p.328]{Mir23}.
The following is an easy consequence of Lemma~\ref{lemmm1}. Note that
this also provides a localized version of Lemma~\ref{lemmm1} if we
consider the case $\eta\equiv 0$.
\begin{cor}\label{cor1} 
For {any} $\theta\in\mathcal{P},$ there exists $\beta\in\mathcal{KL}$ such that for all $\omega\in C([t_0,t_0+\tau), \mathbb{R}_+)$ and all
$\eta\in PC ([t_0,t_0+\tau),\mathbb{R}_+)$ if for given $t_0\ge 0$,
$\tau\in(0,\infty]$ we have that 
\begin{equation*}
    D^{+}\omega(t)\leq-\theta(\omega(t))+\eta(t)
\end{equation*}
holds for all $t\in [t_0,t_0+\tau),$ then for all $t\in[t_0,t_0+\tau)$ it holds that
\begin{equation*}
    \omega(t)\leq\beta(\omega(t_0),t-t_0)+2\int _{t_0}^{t} \eta(s)ds.
\end{equation*}
\end{cor}
The proof here is a straightforward extension of the proof of \cite[Proposition B.4.3, p.240]{Mir23b} and is omitted.

The importance of ISS Lyapunov functions is due to the following
result. The proof follows ideas already developed in \cite[p. 441]{Son89} and
extended to an infinite-dimensional, time-invariant setting in
\cite[Theorem 1.5.4]{Mir23b}.
\begin{thm}\label{hjMMMMMljg}
Let $\Sigma=(X,\mathcal{U},\phi)$ be a time-varying control system
satisfying the BIC property. If there exists a
coercive (L)ISS Lyapunov function $V$ for $\Sigma$, then $\Sigma$ is (L)ISS.
\end{thm}

\begin{proof}
Let $V$ be a coercive LISS Lyapunov function for
$\Sigma=(X,\mathcal{U},\phi)$ and let $\alpha_1,\alpha_2,\kappa,\mu,r_1,r_2$
be as in Definition~\ref{DEFINITION2}. 
Choose $0< \hat{r}_1 < r_1$. 
Define $\rho_1:=\min\{\alpha_2^{-1}\circ\alpha_1(\hat{r}_1),\alpha_1(\hat{r}_1)\}$ and $\rho_2:=\min\{r_2,\kappa^{-1}\circ\alpha_2^{-1}(\rho_1)\}.$
Take an arbitrary control $u\in \mathcal{U},$ with
$\|u\|_{\mathcal{U}}\leq\rho_2,$ and define for $t\geq 0$ 
\begin{equation*}
    \Lambda_{u,t}:=\{x\in D: V(t,x)\leq \alpha_2\circ\kappa(\|u\|_{\mathcal{U}})\}.
\end{equation*}

Since $\|u\|_{\mathcal{U}}\leq\rho_2$,  for any $t\geq0,$ and any $x\in
\Lambda_{u,t},$ it holds using \eqref{R1E} that
\begin{equation}
\label{eq:estimate_rho1}
    \alpha_1(\|x\|_{X}) \leq V(t,x)\leq \alpha_2\circ\kappa(\|u\|_{\mathcal{U}})\leq\alpha_2\circ\kappa(\rho_2)\leq \rho_1,
\end{equation}
and so $\|x\|_{X}\leq \alpha_1^{-1}(\rho_1) \leq \hat{r}_1$. On the other hand,
if $\|x\|_X\leq\kappa(\|u\|_{\mathcal{U}})$ then for all $t\geq 0$ we
have $V(t,x) \leq \alpha_2(\|x\|_X) \leq
\alpha_2\circ\kappa(\|u\|_{\mathcal{U}})$ and so $x\in \Lambda_{u,t}$.
This implies that $ B_{\kappa(\|u\|_{\mathcal{U}})}\subset
\Lambda_{u,t}\subset B_{\hat{r}_1}\subset B_{r_1} \subset D$ for all $t\geq0$.

\noindent {\bf{Step 1:}} We are going to prove that the family $\{\Lambda_{u,t}\}_{t\geq0}$ is forward invariant in the following sense:
\begin{equation}\label{impl1}
 t_0\in\mathbb{R}_+ \wedge x_0\in\Lambda_{u,t_0} \wedge \;
 w\in\mathcal{U}\wedge\;
 \|w\|_{\mathcal{U}}\leq\|u\|_{\mathcal{U}}\qrq \forall t\geq t_0
 : \phi(t,t_0,x_0,w)\in\Lambda_{u,t}.
\end{equation}
Suppose that this implication does not hold for certain $t_0\geq0,$ $x_0\in\Lambda_{u,t_0}$ and $w\in\mathcal{U}$ satisfying
$\|w\|_{\mathcal{U}}\leq\|u\|_{\mathcal{U}}.$ There are two alternatives.

First, it can happen that the maximal existence time $t_m(t_0,x_0,w)$ of the map $\phi(\cdot,t_0,x_0,w)$ is finite and $\phi(t,t_0,x_0,w)\in\Lambda_{u,t}$ for $t\in[t_0,t_m(t_0,x_0,w)).$ 

However, $\Lambda_{u,t}\subset B_{r_1}$ for all $t\geq0$, and thus the trajectory is uniformly bounded on its whole domain of existence.
Then $t_m(t_0,x_0,w)$ cannot be the maximal existence time, as the trajectory can be prolonged to a larger interval thanks to the BIC property of $\Sigma.$

 
The second alternative is that there exist $t\in(t_0,t_m(t_0,x_0,w))$ and $\epsilon>0$ so that 
\[
V(t,\phi(t,t_0,x_0,w))>\alpha_2\circ\kappa(\|u\|_{\mathcal{U}})+\epsilon.
\]
Let
$$\hat{t}:=\inf\{t\in[t_0,t_m(t_0,x_0,w)):\; V(t,\phi(t,t_0,x_0,w))>\alpha_2\circ\kappa(\|u\|_{\mathcal{U}})+\epsilon\}.$$
As $x_0\in\Lambda_{u,t_0}$ and $t\mapsto\phi(t,t_0,x_0,w)$ is
continuous, we obtain with \eqref{R1E} that
\begin{equation}
\label{eq:estimatekappa}
  \alpha_2(\|\phi(\hat{t},t_0,x_0,w)\|_{X})\geq
V(\hat{t},\phi(\hat{t},t_0,x_0,w))=\alpha_2\circ\kappa(\|u\|_{\mathcal{U}})+\epsilon
>\alpha_2\circ\kappa(\|w\|_{\mathcal{U}}).      
\end{equation}
Arguing as in \eqref{eq:estimate_rho1}, we also see
that $\|\phi(\hat{t},t_0,x_0,w)\|_{X} \leq \alpha_1^{-1}(\rho_1 +
\varepsilon)$ and by choosing $\varepsilon>0$ sufficiently small we may assume
$\|\phi(\hat{t},t_0,x_0,w)\|_{X} < r_1$.
Together with \eqref{eq:estimatekappa} this implies that \eqref{STAvNNv}
is applicable and we obtain that  
\begin{equation*}
    D^{+}V(\hat{t},\phi(\hat{t},t_0,x_0,w))\leq -\mu(V(\hat{t},\phi(\hat{t},t_0,x_0,w))).
\end{equation*}

Denote $\myV(t):=V(t,\phi(t,t_0,x_0,w))$, $t\in [0, t_m(t_0,x_0,w))$. Suppose that there is
$(t_k)\subset(\hat{t},t_m(t_0,x_0,w))$, such that $t_k\rightarrow \hat{t}$ as
$k\rightarrow \infty$ and $\myV(t_k)\geq \myV(\hat{t}\,)$ for all
$k\in \N$. It follows that,
$$0 \leq \limsup\limits_{k\to\infty}\frac{\myV(t_k)-\myV(\hat{t}\,)}{t_k-\hat{t}}\leq D^{+}\myV(\hat{t}\,)\leq-\mu(\myV(\hat{t}\,))<0,$$
a contradiction. Thus, there is some $\delta>0$ such that $\myV(t)<\myV(\hat{t}\,)=\alpha_2\circ\kappa(\|u\|_{\mathcal{U}})+\epsilon$ for all $t\in(\hat{t},\hat{t}+\delta).$ This contradicts the definition of $\hat{t}.$ The statement (\ref{impl1}) is shown.

We note for further reference that for any initial time $t_0\geq0$ and any initial condition $x_0\in\Lambda_{u,t_0}$ we have
$$\alpha_1(\|\phi(t,t_0,x_0,u)\|_{X})\leq V(t,\phi(t,t_0,x_0,u))\leq\alpha_2\circ\kappa(\|u\|_{\mathcal{U}}) \quad \forall t\geq t_0,$$
which implies that for all $t_0\geq0$ and all $t\geq t_0$
\begin{equation}\label{inequa1}
\|\phi(t,t_0,x_0,u)\|_{X}\leq\gamma(\|u\|_{\mathcal{U}}),
\end{equation}
where $\gamma=\alpha_1^{-1}\circ\alpha_2\circ\kappa.$

\noindent {\bf{Step 2:}} 
Recall that $u\in \mathcal{U},$ with $\|u\|_{\mathcal{U}}\leq\rho_2.$ Pick
an arbitrary initial time $t_0\geq0$ and any initial condition $x_0$ with
$\|x_0\|_{X}\leq\rho_{1}$. The case $x_0\in \Lambda_{u,t_0}$ has been treated in Step 1 of this proof. In Step 2, we consider the case $x_0\notin\Lambda_{u,t_0}$.  Due to the continuity of the flow map $t\mapsto\phi(t,t_0,x_0,u)$ and of the Lyapunov function $V$, there exists $\tau>0$ such that $\phi(t,t_0,x_0,u)\notin\Lambda_{u,t}$ for all $t\in [t_0,t_0+\tau].$

Using the axiom of shift invariance (Definition~\ref{csyol}\,(iii)) for $\Sigma$ and employing the
sandwich bounds \eqref{R1E}, we get 
\begin{equation*}
\alpha_2(\|\phi(t,t_0,x_0,u)\|_{X})\geq
V(t,\phi(t,t_0,x_0,u))\geq\alpha_2\circ\kappa(\|u\|_{\mathcal{U}}),\quad t
\in [t_0,t_0+\tau],
\end{equation*}
and using that $\alpha_2\in\mathcal{K}_{\infty},$ we obtain that $\|\phi(t,t_0,x_0,u)\|_{X}\geq\kappa(\|u\|_{\mathcal{U}})$ for $t\in[t_0,t_0+\tau].$

As long as $\|\phi(t,t_0,x_0,u)\|_{X}\leq r_1,$ and $t \in
[t_0,t_0+\tau]$,  we obtain from \eqref{STAvNNv} that
\begin{equation}\label{INE11}
D^{+}V(t,\phi(t,t_0,x_0,u))=\dot{V}_{u}(t,\phi(t,t_0,x_0,u))\leq -\mu(V(t,\phi(t,t_0,x_0,u))).
\end{equation}
Then, for all such $t,$ it holds that $V(t,\phi(t,t_0,x_0,u))\leq V(t_0,x_0),$ which implies that
$$\|\phi(t,t_0,x_0,u)\|_{X}\leq\alpha_1^{-1}\circ\alpha_2(\|x_0\|_{X})\leq\alpha_1^{-1}\circ\alpha_2(\rho_1)\leq r_1.$$
Thus, (\ref{INE11}) is valid for all $t\in [t_0,t_0+\tau].$\\
Appealing to Lemma~\ref{lemmm1} and
Corollary~\ref{cor1} with $\eta\equiv 0$, there exists a $\hat{\beta}\in
\mathcal{KL}$ such that the estimate
\eqref{eq:conclusionlemma1} holds on $[t_0,t_0+\tau)$. 
Since the map $t \mapsto V(t,\phi(t,t_0,x_0,u))$ is continuous, we infer that
\begin{equation*}
    V(t,\phi(t,t_0,x_0,u))\leq
\hat{\beta}(V(t_0,x_0),t-t_0),\quad t\in[t_0,t_0+\tau]
\end{equation*}
and defining $\beta(s,t):=\alpha_1^{-1}\circ
\hat{\beta}(\alpha_2^{-1}(s),t)$, $s,t\geq0$, we have
\begin{equation}\label{EE11}
\|\phi(t,t_0,x_0,u)\|_{X}\leq \beta(\|x_0\|_{X},t-t_0),\quad  \quad  t\in[t_0,t_0+\tau].
\end{equation}
The estimate (\ref{EE11}) together with the BIC property ensure that the solution can be
prolonged to an interval that is larger than $[t_0,t_0+\tau],$ and if on this interval $\phi(t,t_0,x_0,u) \notin \Lambda_{u,t},$ then the estimate (\ref{EE11}) holds on this larger interval. This shows
that $\phi(t,t_0,x_0,u)$ exists at least until the time when it intersects $\Lambda_{u,t}$.
As $B_{\kappa(\|u\|_{\mathcal{U}})}\subset \Lambda_{u,t}$ for all
$t\geq0$, the estimate \eqref{EE11} implies that there exists $0<\bar{t}<\infty$ such that
\begin{equation*}
    \bar{t}=\inf\{t\geq t_0:\phi(t,t_0,x_0,u)\in \Lambda_{u,t}\}.
\end{equation*}
By the above reasoning, \eqref{EE11} holds on $[t_0,\bar{t}]$.

From the forward invariance of the family  $\{\Lambda_{u,t}\}_{t\geq0}$,
the estimate (\ref{inequa1}) is valid for all $\bar{t}\leq t <
t_m(t_0,x_0,u)$ and by the BIC property we have $t_m(t_0,x_0,u)= \infty$. From (\ref{inequa1}) and (\ref{EE11}), we conclude that
\begin{equation*}
\|\phi(t,t_0,x_0,u)\|_{X}\leq\beta(\|x_0\|_{X},t-t_0)+\gamma(\|u\|_{\mathcal{U}})\quad  \forall t\geq t_0.    
\end{equation*}
In our reasoning $t_0\geq0$, $x_0 \in B_{\rho_1}$, and $u \in {\cal U}$
with $\|u\|_{\mathcal{U}}\leq\rho_2$ were arbitrary and the constructed
functions $\beta,\gamma$ did not depend on the specific choice of these variables. Thus, combining (\ref{inequa1}) and (\ref{EE11}), we obtain the LISS estimate for $\Sigma$ in $(t,t_0,x_0,u)\in[t_0,\infty)\times\mathbb{R}_+\times B_{\rho_1}\times B_{\rho_2,\mathcal{U}}.$

To prove that the existence of an ISS Lyapunov function ensures that
$\Sigma$ is ISS, we argue as above but with $r_1=r_2=\infty.$
\end{proof}

We will now define the notion of a (non)-coercive iISS Lyapunov function and prove that the existence of a coercive iISS (resp. coercive ISS) Lyapunov function implies iISS (resp. ISS).
\begin{dfn}
Consider a time-varying control system $\Sigma=(X,\mathcal{U},\phi)$ with the input space $\mathcal{U}=PC(\mathbb{R}_+, U).$
A continuous function $V:\mathbb{R}_+\times X \to \mathbb{R}_+$ is called
a non-coercive iISS Lyapunov function for $\Sigma$, if there exist
$\alpha_2\in \mathcal{K}_{\infty},$ $\eta\in \mathcal{P}$ and $\chi\in
\mathcal{K},$ such that for all $t\geq 0,$ $V(t,0)=0,$ the bounds \eqref{R1EN} hold,
and the Lie derivative along the trajectories of the system $\Sigma$
satisfies for all $x\in X,$ all $u\in \mathcal{U},$ and all $t\geq 0$
\begin{equation}\label{R2E}
\dot{V}_{u}(t,x)\leq -\eta(\|x\|_{X})+\chi(\|u(t)\|_{U}).
\end{equation}
If, in addition, there is $\alpha_1\in\mathcal{K}_{\infty},$ such that (\ref{R1E}) holds,
then $V$ is called a (coercive) iISS Lyapunov function for $\Sigma.$
If
\begin{equation}\label{R1Ev}
\displaystyle\lim _{\tau \to \infty} \eta(\tau)=\infty \; \;\text{ or } \; \;\liminf _{\tau \to \infty}\eta(\tau)\geq \lim _{\tau \to \infty}\chi(\tau),
\end{equation}
 then $V$ is called a (coercive or noncoercive) ISS Lyapunov function for $\Sigma$ in a dissipative form.
\end{dfn}

The next theorem underlines the importance of iISS and ISS Lyapunov functions.
\begin{thm}\label{hjljg}
Let $\mathcal{U}=PC(\mathbb{R}_+, U)$.
Consider a time-varying control system $\Sigma=(X,\mathcal{U},\phi),$ satisfying the BIC property. If there is a coercive \emph{iISS} (resp. coercive ISS) Lyapunov function in dissipative form for $\Sigma,$ then $\Sigma$ is iISS (resp. ISS).
\end{thm}

\begin{proof}
Let $V$ be a coercive iISS Lyapunov function with associated functions $\alpha_1,\alpha_2,\eta,\chi$.
As $\eta\in \mathcal{P},$ it follows from \cite[Proposition
B.1.14]{Mir23b} that there are $\varrho\in\mathcal{K}_{\infty}$ and
$\xi\in\mathcal{L}$ such that $\eta(r)\geq\varrho(r)\xi(r)$ for all
$r\geq0.$ Pick any $\tilde{\eta}\in\mathcal{P}$ satisfying
\begin{equation*}
    \varrho(\alpha_2^{-1}(r))\xi(\alpha_1^{-1}(r))\geq \tilde{\eta}(r), \quad r\geq 0.
\end{equation*}
From (\ref{R1E}) and (\ref{R2E}), we obtain for any $x\in X,$ any $t\geq0$ and any $u\in\mathcal{U}$ that
\begin{align*}
\dot{V}_{u}(t,x)&\leq -\eta(\|x\|_{X})+\chi(\|u(t)\|_{U})\\
&\leq-\varrho(\|x\|_{X})\xi(\|x\|_{X})+\chi(\|u(t)\|_{U}).\\
\intertext{Using $\alpha_2^{-1}(V(t,x)) \leq \|x\|_{X} \leq
  \alpha_1^{-1}(V(t,x))$ we can continue } 
&\leq-\varrho\circ\alpha_2^{-1}(V(t,x))\xi\circ\alpha_1^{-1}(V(t,x))+\chi(\|u(t)\|_{U})\\
&\leq-\tilde{\eta}(V(t,x))+\chi(\|u(t)\|_{U}).
\end{align*}
Now take any input $u\in\mathcal{U},$ any initial time $t\geq0,$ any initial condition $x_0\in X,$ and consider the corresponding solution $\phi(\cdot,t_0,x_0,u)$ defined on $[t_0,t_m(t_0,x_0,u)).$

As $V$ is continuous and $\phi$ is continuous in $t,$ the map $t\mapsto V(t,\phi(t,t_0,x_0,u))$ is continuous, and for all $t_0\geq0$ and all $t\in[t_0,t_m(t_0,x_0,u)),$ it holds that
$$D^{+}V(t,\phi(t,t_0,x_0,u))\leq-\tilde{\eta}(V(t,x))+\chi(\|u(t)\|_{U}).$$
By Corollary \ref{cor1}, there exists a $\hat{\beta}\in \mathcal{KL},$ so that
$$V(t,\phi(t,t_0,x_0,u))\leq \hat{\beta}(V(t_0,x_0),t-t_0)+2\int _{t_0} ^{t} \chi(\|u(s)\|_{U})ds, \quad t_0\geq0, \quad t\in[t_0,t_m(t_0,x_0,u)).$$
With \eqref{R1E}, it holds that
\begin{equation}\label{wxcvs}
\alpha_1(\|\phi(t,t_0,x_0,u)\|_{X})\leq \tilde{\beta}(\|x_0\|_{X},t-t_0)+2\int _{t_0} ^{t} \chi(\|u(s)\|_{U})ds, \quad t_0\geq0, \quad t\in[t_0,t_m(t_0,x_0,u)),
\end{equation}
where  $\tilde{\beta}\in \mathcal{KL}$ is defined by
$\tilde{\beta}(s,t)=\hat{\beta}(\alpha_2(s),t)$, $s,t\geq0$.
Since $\alpha_1 \in\mathcal{K}_{\infty}$, it satisfies the triangle
inequality $\alpha_1^{-1}(a+b)\leq\alpha_1^{-1}(2a)+\alpha_1^{-1}(2b)$ for
all $a,b\in\mathbb{R}_{+}$. We conlude the system is iISS as in Definition \ref{rhan23} with $\alpha \in \mathcal{K}_{\infty},$ $\mu\in \mathcal{K}$ and $\beta\in\mathcal{KL},$ with $\alpha(s)=\alpha_1^{-1}(2s),$ $\mu(s)=2 \chi(s)$ and $\beta(s,r)=\alpha_1^{-1}(2\tilde{\beta}(s,r)).$
Estimate (\ref{wxcvs}) together with the BIC property, implies that
$t_m(t_0,x_0,u)= \infty$ for all $(t_0,x_0,u) \in \R_+\times X \times {\cal U}$ and hence
system $\Sigma$ is iISS.

Suppose that $V$ is an ISS Lyapunov function in dissipative form, i.e., (\ref{R2E}) holds. Arguing as in the proof of \cite[Theorem 3.2]{damak2021input}, we obtain that
$V$ is an ISS Lyapunov function of $\Sigma$ in implication
form. Consequently, $\Sigma$ is ISS by Theorem~\ref{hjMMMMMljg}.
\end{proof}

\section{Time-varying semi-linear evolution equations}
\label{sec:time-varying-semi}

In the previous section, a general ISS formalism for abstract control
systems has been  introduced. We now proceed to the analysis of time-varying semi-linear systems in Banach spaces.

Let $X$ and $U$ be Banach spaces, and the input space be $\mathcal{U}=PC(\mathbb{R}_{+}, U).$ We consider time-varying semi-linear evolution equations of the form:
\begin{equation}
\label{R1}
\left\lbrace
\begin{array}{l}
\dot{x}(t)=A(t)x(t)+\Psi(t,x(t),u(t)),\qquad t\ge t_0\ge0,\\
x(t_0)=x_0,
\end{array}\right.
\end{equation}
where $x\in X$ is the state of the system, $u\in \mathcal{U}$ is a
control input, $t_0 \ge 0$ is an initial time, $x_0$ is the initial
state, $\Psi:\mathbb{R}_{+}\times X\times U \to X$ is a nonlinear map for
which we will specify the assumptions below. In addition,
$\{A(t):D(A(t)) \subset X \rightarrow X\}_{t \geq t_0\geq 0}$ is a family
of linear operators such that for all $t\geq 0,$ the domain $D(A(t))$ of
$A(t)$ is
 dense in $X.$ 

We assume that
$\{A(t)\}_{t\geq0}$ generates a strongly continuous evolution family
$\{W(t,s)\}_{t\geq s\geq0},$ that is, for all $t\geq s\geq t_0\ge0$ there
exists a bounded linear operator $W(t,s):X\to X$ and the following
properties hold:
\begin{enumerate}
\item[$(i)$] $W(s,s)=I,$ \quad $W(t,s)=W(t,r)W(r,s)$ for all $t\geq r\geq s\geq 0.$
\item[$(ii)$] $(t,s)\mapsto W(t,s)$ is strongly continuous for $t\geq
  s\geq 0,$ that is for each $x \in X$ the mapping ${\cal T} \ni(t,s)\mapsto W(t, s)x$ is continuous.
\item[$(iii)$] For all $t\geq s\geq 0$ and all $\upsilon\in D(A(s)),$ we
  have $W(t,s)v \in D(A(t))$ and 
$$\frac{\partial}{\partial t}\big(W(t,s)\upsilon\big)=A(t)W(t,s)\upsilon,$$
$$\frac{\partial}{\partial s}\big(W(t,s)\upsilon\big)=-W(t,s)A(s)\upsilon.$$
\end{enumerate}
\begin{rem}
Provided it exists, the evolution family $\{W(t,s)\}_{t\geq s\geq0}$ is uniquely determined by the family $\{A(t)\}_{t\geq0},$ see, e.g., \cite[p. 58]{carmen}.
\end{rem}
In the next definition, we extend the concept of a classical solution, as defined in \cite[Definition 2.1, p. 105]{Paz83}, to encompass piecewise classical solutions for \eqref{R1}.
\begin{dfn}\label{rrnnann23}
Let $x_0\in X,$ $u\in\mathcal{U}$ and $t_0\ge 0$ be given. A function
$x:[t_0,\infty)\rightarrow X$ is called a \emph{piecewise classical
  solution} of (\ref{R1}) on $[t_0,\infty)$ if $x(\cdot)$ is continuous on
$[t_0,\infty),$ piecewise continuously differentiable\footnote{To be
    precise, there is a countable, locally finite partition $[t_0,\infty)
    = \cup_{j=0}^\infty [a_j,a_{j+1})$, $a_0< a_1< a_2 < \ldots$, such
    that $x(\cdot)$ is continuously differentiable on each open interval $(a_j,a_{j+1})$.} on $(t_0,\infty)$ and $x(t)$ satisfies (\ref{R1}) 
at all points of differentiability.
\end{dfn}

\subsection{Semilinear evolution equations as control systems}
In this section, we show under appropriate Lipschitz continuity
assumptions that the semi-linear system (\ref{R1}) generates a well-posed
control system in the sense defined in Section~\ref{sec:time-varying-control}.
\begin{dfn}
We call $\Psi:\mathbb{R}_+\times X\times U\longrightarrow X$ \emph{locally
  Lipschitz continuous in $x,$ uniformly in $t$ and $u$ on bounded
  subsets} if for every $\tilde{t}\geq0$ and $c\geq 0,$ there is a
constant $K(c,\tilde{t}\,),$ such that for all $x,y\in
X:\|x\|_{X},\|y\|_{X}\leq c$, all $t\in [0,\tilde{t}]$, and all
$u\in U$: $\|u\|_{U}\leq c$, it holds that
\begin{equation}\label{bvbnbpp}
\|\Psi(t,x,u)-\Psi(t,y,u)\|_{X}\leq K(c,\tilde{t}\,)\|x-y\|_{X}.
\end{equation}
\end{dfn}

The following assumption will be needed in the remainder of the paper
\begin{description}
\item[$(\mathcal{H}_1)$] The family $\{A(t)\}_{t\geq0}$ generates a strongly continuous evolution family
$\{W(t,s)\}_{t\geq s\geq0}$. The nonlinearity $\Psi$ is jointly continuous in $(t,u)$ and locally Lipschitz continuous in $x$, uniformly in $t$ and $u$ on bounded sets.
\end{description}
Now, we introduce the concept of mild solutions for (\ref{R1}) and explore some associated properties.
\begin{dfn}
Assume $(\mathcal{H}_1)$. Let $0 \leq t_0 < t_1 \leq \infty$. A continuous function
$x:[t_0,t_1)\rightarrow X$ is called a \emph{(mild) solution} of
(\ref{R1}) on the interval $[t_0,t_1)$ for the initial condition
$x(t_0)=x_0\in X,$ and the input $u\in \mathcal{U}$, if $x$ satisfies the integral equation
\begin{equation}\label{bvbbLL4}
x(t) = W(t,t_0)x_0 + \int _{t_0} ^{t} W(t,s)\Psi(s,x(s),u(s))ds, \quad t
\in [ t_0,t_1 ).
\end{equation}
\end{dfn}
Let $u\in {\cal U}$ be fixed and consider solutions $x_1, x_2$ of (\ref{R1}) defined on  intervals $[t_0, \tau_1)$
and $[t_0, \tau_2),$ respectively, $\tau_1, \tau_2 > t_0.$ We call $x_2$ an extension of $x_1$ if $\tau_2 > \tau_1,$ and
$x_2(t) = x_1(t)$ for all $t \in [t_0, \tau_1).$
A mild solution $x$ of (\ref{R1}) is called \emph{maximal}, if there is no
solution of (\ref{R1}) that extends $x$. A solution is called
\emph{global} if it is defined on $[t_0,\infty).$

A central property of the system (\ref{R1}) is
\begin{dfn}\label{phi}
Assume $(\mathcal{H}_1)$. The system (\ref{R1}) is called
\emph{well-posed} if for every initial time $t_0\geq0$, every initial
condition $x_0\in X$ and every control input $u \in \mathcal{U},$ there
exists a unique maximal mild solution that we denote by 
\begin{equation*}
 \phi(\cdot,t_0,x_0, u):[t_0, t_m(t_0,x_0,u))\to X.
\end{equation*}
In this case $t_m=t_m(t_0,x_0,u)$ is called the \emph{maximal existence time} of the solution corresponding to $(t_0,x_0,u).$
\end{dfn}
The map $\phi$ defined in Definition \ref{phi}, describing the evolution
of the system (\ref{R1}), is called \emph{the flow map}, or just
\emph{flow}. The domain of definition of the flow $\phi$ is denoted 
\begin{equation*}
    D_{\phi}\subseteq\mathcal{T}\times X\times\mathcal{U}.
\end{equation*}
The following result gives sufficient conditions for the well-posedness of
the system (\ref{R1}). It is a minor extension of a result by
Pazy, \cite[Theorem 6.1.4, p.185]{Paz83}, to the situation when the inputs are piecewise continuous.
\begin{prop}
\label{propooo1}
If Assumption $(\mathcal{H}_1)$ holds, then system (\ref{R1}) is a well-posed control system with the BIC property.
\end{prop}
\begin{proof}
Let $t_0\geq 0,$ $x_0\in X,$ and $u\in PC(\mathbb{R}_{+}, U).$ Since $u$ is piecewise right continuous, there is an increasing sequence of discontinuities $(t_k)_{k\in \mathbb{N}}$ of $u$ without accumulation points.
We will construct the solution in an iterative fashion. Extend $u_{\vert [t_0,t_1)}$ in an arbitrary manner to a continuous function $\tilde{u}_0$ defined on
$[t_0,t_1+1]$ with $u_{\vert [t_0,t_1)} = \tilde{u}_{0,\vert [t_0,t_1)}$.  By an application
of the existence result \cite[Theorem 6.1.4, p.185]{Paz83}, there is $\tilde{t}\in (t_0,t_1+1),$ such that there is a unique
(maximal) mild solution of (\ref{R1}) on $[t_0,\tilde{t}),$ corresponding
to the initial condition $x_0$ and the input $\tilde{u}_0$.
If $\tilde{t}\leq t_1$, then $\tilde{t}$ is the maximal existence time
corresponding to $(t_0,x_0,u)$, as $u$ and $\tilde{u}_0$ coincide on
$[t_0,\tilde{t})$. Thus
$t_m=t_m(t_0,x_0,u)=\tilde{t}$. Otherwise, $x_1 :=
\phi(t_1,t_0,x_0,\tilde{u}_0) = \phi(t_1,t_0,x_0,u)$ exists. We may then
repeat the argument on the interval $[t_1, t_2 +1 ]$, again using a
continuous extension $\tilde{u}_1$ of $u_{\vert [t_1,t_2]} $ to that interval. 

Let $x: [t_0, \min \{ t_2 , t_m(t_1,x_1,\tilde{u}_1))\to X$ be defined by 
\begin{eqnarray*}
x(t):=
\begin{cases}
\phi(t,t_0,x_0,u), \quad & t \in [ t_0,t_1 ],\\
\phi(t,t_1,x_1,\tilde{u}_1) = \phi(t,t_1,x_1,u),\quad & t \in (t_1,\min \{ t_2 , t_m(t_1,x_1,\tilde{u}_1)).
\end{cases}
\end{eqnarray*}
In the equality in the definition of $x$, we have used that $u$ and $\tilde{u}_1$ coincide on $[t_1,t_2]$.

Then $x$ is continuous, clearly a solution of the integral equation on
$[t_0,t_1 ]$ and we have for $t_1 \leq t < \min \{ t_2 , t_m
\}$ that
\begin{align*}
    x(t) &= W(t,t_1) x_1 + \int_{t_1}^t
    W(t,s)\Psi(s,\phi(s,t_1,x_1,u),u(s)) ds \\
		&=  W(t,t_1) \left(
      W(t_1,t_0)x_0 + \int_{t_0}^{t_1}W(t_1,s)\Psi(s,\phi(s,t_0,x_0,u),u(s))ds \right)  + \int_{t_1}^t
    W(t,s)\Psi(s,x(s),u(s)) ds \\
		&=  W(t,t_0) x_0 + \int_{t_0}^{t}W(t,s)\Psi(s,x(s),u(s))ds.
\end{align*}
This shows that indeed $x$ is a mild solution on $[t_0, \min \{ t_2 , t_m(t_1,x_1,\tilde{u}_1))$.
By applying this reasoning iteratively, we obtain a unique maximal mild solution
$\phi(\cdot,t_0,x_0,u)$ on an interval $[t_0,t_m)$. If $u$ only has finitely many
discontinuities the same arguments can be applied on the final infinite interval of
continuity. By construction and \cite[Theorem 6.1.4, p.185]{Paz83}, each solution can be continued as long it
remains bounded, so that the BIC property is guaranteed.
\end{proof}
Now we show that well-posed systems of the form (\ref{R1}) are a special case of general control systems introduced in Definition \ref{csyol}.
\begin{thm}
Let (\ref{R1}) be well-posed. Then the triple $(X,\mathcal{U},\phi)$, where $\phi$ is the flow map of (\ref{R1}), constitutes a control system in the sense of Definition \ref{csyol}.
\end{thm}
\begin{proof}
By construction, conditions (i)--(iv) as well as
($\Sigma_1$)--($\Sigma_3$) in Definition~\ref{csyol} are
satisfied. We only need to check the cocycle property ($\Sigma_4$). Let $\phi$ be the
flow map and consider an initial time $t_0\geq 0$, an initial condition
$x_0\in X$, and an input $u \in \mathcal{U}$.
Let $t \in (t_0,t_m(t_0,x_0,u))$, $\tau\in [t_0,t]$ and consider the solution corresponding to the
initial condition $x(\tau) = \phi(\tau,t_0,x_0,u)$ and the input $u$. By
well-posedness of the system, there exists a unique solution corresponding
to these data on a maximal time interval $[\tau, \tau')$. Define a function $x
: [t_0, \min \{\tau',t \})\to X$, by $x(s) = \phi(s,t_0,x_0,u)$, $s\in [t_0,\tau ]$,
and $x(s) =  \phi(s,\tau,\phi(\tau,t_0,x_0,u),u)$, $s\in [\tau, \min \{\tau',t \})$. Then $x$ is continuous, clearly solves the integral equation
(\ref{bvbbLL4}) on the interval $[t_0,\tau ]$ and for $s \in [\tau, \min
\{\tau',t \})$ we have
\begin{multline*}
    x(s) = \phi(s,\tau,\phi(\tau,t_0,x_0,u),u) =
    W(s,\tau)\phi(\tau,t_0,x_0,u)+\int_{\tau}^{s}W(s,r)\Psi(r,\phi(r,\tau,\phi(\tau,t_0,x_0,u),u),u(r))dr
    \\ 
    = W(s,\tau)\bigg{[}W(\tau,t_0)x_0+\int_{t_0}^{\tau}W(\tau,r)\Psi(r,\phi(r,t_0,x_0,u),u(r))dr \bigg{]}
+\int _{\tau}^{s} W(s,r)\Psi(r,x(r),u(r))dr \\
= W(s,t_0)x_0+\int _{t_0}^{s} W(s,r)\Psi(r,x(r),u(r))dr.
\end{multline*}

This shows that $x$ is a solution for the initial values $(t_0, x_0)$ and
the input $u$ on the interval  $[t, \min \{\tau',t \})$. As this solution
is unique, $x$ coincides with $\phi(\cdot, t_0,x_0,u)$ on $[t, \min
\{\tau',t \})$. The BIC property then implies that $\tau' > t$
and the cocycle property holds, because  $
\phi(t,\tau,\phi(\tau,t_0,x_0,u),u)= x(t) = \phi(t,t_0,x_0,u)$.
\end{proof}

\subsection{Stability of evolution families}
In this subsection, we recall different types of stability for strongly continuous evolution families $\{W(t,t_0)\}_{t\geq t_0\geq0}$. These stability concepts include uniform stability, uniform attractiveness, uniform asymptotic stability, and uniform exponential stability. We also provide a proposition and an example that help to clarify the relationships between these different forms of stability.

\begin{dfn}
\label{def:StabAttr}
\cite[p. 112]{DaK74} \cite[Def 36.9, p. 174]{Hah67} \cite[Def 3.4, p. 60]{carmen}
A strongly continuous evolution family $\{W(t,t_0)\}_{t\geq t_0\geq0}$ is said to be:
\begin{enumerate}
\item[$(i)$] \emph{Uniformly stable} if there is $0<N<\infty$ such that
  for all $t\geq t_0 \geq0$ we have $\|W(t,t_0)\|\leq N.$
\item[$(ii)$] \emph{Uniformly attractive} if for all $\epsilon>0$, there exist $T=T(\epsilon)$ such that, for all $t_0\geq0$, and all $t\geq t_0+T$ we have $\|W(t,t_0)\| \leq \epsilon$.
\item[$(iii)$] \emph{Uniformly asymptotically stable} if $\{W(t,t_0)\}_{t\geq t_0\geq0}$ is uniformly stable and uniformly attractive.
\item[$(iv)$] \emph{Uniformly exponentially stable} if there exist $k,\omega>0,$ such that $\|W(t,t_0)\|\leq ke^{-\omega(t-t_0)} $ holds for each $t_0\geq0$ and all $t\geq t_0$.
\end{enumerate}
\end{dfn}

The following lemma is essential for establishing a fundamental connection between the stability of evolution families and the behavior of time-varying linear systems.
\begin{lem}
\label{lemmmcc1}\cite[Lemma $4.1$]{damak2021input}
Let $\{W(t,s)\}_{t\geq s\geq0}$ be a strongly continuous evolution family
with generating family $\{ A(t) \}_{t\geq 0}$. 
The following statements are equivalent:
\begin{enumerate}
\item[$(i)$] The family $\{W(t,s)\}_{t\geq s\geq0}$ is uniformly asymptotically stable.
\item[$(ii)$] The time-varying linear system \begin{equation}\label{linear}
\dot{x}=A(t)x,\quad x(t_0)=x_0,
\end{equation}
 is uniformly exponentially stable, that is, there exist $k,\omega > 0,$ such that
$$\|x(t)\| \leq k\|x_0\|e^{-\omega(t-t_0)}$$ holds for all $t_0\geq0,$ all $x_0 \in X,$ and all $t \geq t_0.$

\item[$(iii)$]The family $\{W(t,s)\}_{t\geq s\geq0}$ is uniformly exponentially stable.
\end{enumerate}
\end{lem}

In the rest of this section and in 
Example~\ref{ex:Uniform-attractivity-and-uniform-stability} in the
Appendix, we clarify the
relationship between uniform attractivity and uniform asymptotic stability. This is presumably well-known  in the theory of linear
time-varying systems. Nevertheless, we could not find the statements of
these results in the literature, thus we state them for completeness. 

In addition to (global in time) stability properties, let us introduce
the following concept characterizing bounds of $W$ on finite time
intervals. The definition is motivated by 
\cite[Definition 1.4, p. 30]{KaJ11}, \cite{Mir23e}.
\begin{dfn}
\label{def:UBRS}
A strongly continuous evolution family $\{W(t,t_0)\}_{t\geq t_0\geq0}$ has
\emph{uniformly (in time) bounded reachability sets (UBRS)} if for any $T>0$ it holds that 
\begin{eqnarray}
\sup_{t_0\geq 0,\ t\in[t_0,t_0+T]}\|W(t,t_0)\|<\infty.
\label{eq:UBRS}
\end{eqnarray}
\end{dfn}
We can characterize UBRS as follows:
\begin{lem}
\label{lem:UBRS-criterion}
A strongly continuous evolution family $\{W(t,t_0)\}_{t\geq t_0\geq0}$ is
UBRS if and only if
\begin{eqnarray}
\sup_{t_0\geq 0,\ t\in[t_0,t_0+1]}\|W(t,t_0)\|=K<\infty.
\label{eq:Finite-Bohl-exponent}
\end{eqnarray}
\end{lem}

\begin{proof}
Clearly, UBRS implies \eqref{eq:Finite-Bohl-exponent}.

Conversely, let \eqref{eq:Finite-Bohl-exponent} hold. Pick any $T>0$. By taking a larger $T$ if needed, we can assume that $T \in \mathbb{N}.$

Using \eqref{eq:Finite-Bohl-exponent} and the property $(i)$ of $\{W(t,t_0)\}_{t\geq t_0\geq0}$, we have for any $r\in\{1,\ldots,T\}$, 
and any $t\in [t_0+r-1,t_0+r]$ that
$$ 
W(t,t_0)=W(t,t_0+r-1)W(t_0+r-1,t_0+r-2)...W(t_0+2,t_0+1)W(t_0+1,t_0),
$$
and then for any $t\in [t_0+r-1,t_0+r]$
\begin{align*}
\|W(t,t_0)\| &\leq  \|W(t,t_0+r-1)\|\|W(t_0+r-1,t_0+r-2)\|\\
& \qquad \ldots \|W(t_0+2,t_0+1)\|\|W(t_0+1,t_0)\|< K^r<\infty.
\end{align*}
As $K\geq 1$ (consider $t=t_0$ in \eqref{eq:Finite-Bohl-exponent}), this ensures that
\begin{eqnarray*}
\sup_{t_0\geq 0,\ t\in[t_0,t_0+T]}\|W(t,t_0)\|\leq K^T<\infty.
\end{eqnarray*}
This finishes the proof.
\end{proof}

Now we can characterize uniform asymptotic stability in terms of uniform attractivity and UBRS.

\begin{prop}
\label{prop:Uniform-attractivity-and-uniform-stability}
$\{W(t,t_0)\}_{t\geq t_0\geq0}$ is uniformly asymptotically stable if and only if 
$\{W(t,t_0)\}_{t\geq t_0\geq0}$ is uniformly attractive and UBRS.
\end{prop}

\begin{proof}
Let $\kappa>0.$ By the uniform attractivity of $\{W(t,t_0)\}_{t\geq t_0\geq0}$, there exists $T=T(\kappa)>0,$ such that $\|W(t,t_0)\|\leq \kappa$, for all $t_0\geq0$ and all $t\geq t_0+T.$ 

By UBRS, there is $K>0$ so that
\begin{eqnarray}
\sup_{t_0\geq 0,\ t\in[t_0,t_0+T]}\|W(t,t_0)\|\leq K<\infty.
\label{eq:Finite-Bohl-exponentT}
\end{eqnarray}
Choosing $ N:=\max\left(K,\kappa\right) $, we obtain $\|W(t,t_0)\|\leq
N$, for all $t\geq t_0\geq0$. Hence, $\{W(t,t_0)\}_{t\geq t_0\geq0}$ is
uniformly stable, and overall it is uniformly asymptotically stable.
The converse implication is obvious.
\end{proof}

In view of \cite[Theorem 4.2]{DaK74}, the condition \eqref{eq:Finite-Bohl-exponent} is equivalent to the finiteness of the upper Bohl exponent of \eqref{linear}. This condition holds, e.g., if $A(\cdot)$ is continuous and uniformly bounded on $[0,\infty)$.
Without finiteness of the upper Bohl exponent,
Proposition~\ref{prop:Uniform-attractivity-and-uniform-stability} does not
necessarily hold, in contrast to the time-invariant linear systems, where
strong stability of a $C_0$-semigroup (i.e. non-uniform global
attractivity of a dynamical system) always implies uniform stability. We
illustrate this fact by means of Example
\ref{ex:Uniform-attractivity-and-uniform-stability} in the Appendix.

\section{Lyapunov criteria for ISS of time-varying linear systems}
\label{sec:lyap-crit-iss}

\

We now consider a special case of (\ref{R1}), namely linear systems on a Banach space $X$ of the form: 
\begin{equation}\label{unb}
\left\lbrace
\begin{array}{l}
\dot{x}(t)=A(t)x(t)+B(t)u(t), \qquad t\ge t_0\ge0,\\
x(t_0)=x_0,
\end{array}\right.
\end{equation}
where $B\in C(\mathbb{R}_{+},L(U,X)),$ with $\displaystyle
\sup_{t\geq0}\|B(t)\|<\infty.$ We assume that the family $\{ A(t)
\}_{t\geq 0}$ generates a strongly continuous evolution family $\{W(t,s)\}_{t\geq s\geq0}$. As in the general nonlinear case, we assume that inputs belong to the space $\mathcal{U}=PC(\mathbb{R}_{+},U).$
Since $B$ is continuous and $u$ is piecewise right continuous, $\Psi:(t,x,u)\mapsto B(t)u(t)$ satisfies the assumption $(\mathcal{H}_1)$ and, according to Proposition \ref{propooo1}, there is a unique global mild solution $\phi(\cdot,t_0,x_0,u)$ of system (\ref{unb}) for any data $(t_0,x_0,u).$
By definition, it has the form
\begin{equation}\label{b5484}
\phi(t,t_0,x_0,u)= W(t,t_0)x_0+\int_{t_0}^{t}W(t,\tau)B(\tau)u(\tau)d\tau.
\end{equation}

\subsection{Uniformly bounded and continuous $A$}
\label{sec:unif-bound-cont}

In this part, we concentrate on the case of Hilbert spaces $X$ endowed with the scalar product $\langle \cdot,\cdot\rangle.$
We start with a Lyapunov characterization of the ISS property for the
system \eqref{unb} for the case that $A: \R_+ \to L(X)$ is continuous in
the uniform operator topology. Under this assumption, the mild solutions
of \eqref{unb} exist and are piecewise classical solutions in the sense of Definition \ref{rrnnann23}.

\begin{prop}\label{prop:mild-is-classical}
Let $A : \R_+ \to L(X)$ be continuous and uniformly bounded in the uniform operator topology. Then, for every $u\in \mathcal{U},$ every initial time $t_0$ and every initial condition $x_0,$ the 
mild solution $\phi(\cdot,t_0,x_0,u)$ given by (\ref{b5484}) belongs to $PC^{1}([t_0,\infty),X)$ and is the unique piecewise classical solution of (\ref{unb}).
\end{prop}

\begin{proof}
By \cite[Theorem 5.1, p. 127]{Paz83}, for each initial time $t_0$ and each initial condition $x_0$, there is a 
unique classical solution $x(\cdot)$ for the (undisturbed) system (\ref{linear}). Defining
$$W(t,t_0)x_0:=x(t),\quad t\geq t_0,$$
we construct an evolution family $\{W(t,t_0)\}_{t\geq t_0\geq0}: X \to  X$ of bounded operators, generated by $\{A(t)\}_{t\geq0}.$ Furthermore, according to \cite[Theorem 5.2, p. 128]{Paz83}, we have the following formulas for the partial derivatives of $W$ taken in $L(X):$
$$\frac{\partial}{\partial t}W(t,s)=A(t)W(t,s),\quad \frac{\partial}{\partial s}W(t,s)=-W(t,s)A(s).$$
The result follows using arguments as in \cite[Equation (1.18), p. 129]{Paz83}.
\end{proof}
For the Lyapunov analysis of linear systems, we recall the following notions.
A self-adjoint operator $S \in L(X)$ is said to be positive, if $\langle Sx,x\rangle >0$ for all $x\in X\backslash\{0\}.$ We call an operator-valued function $P:\mathbb{R}_{+}\to L(X)$ positive if $P(t)$ is self-adjoint and positive for any $t\geq0$.
An operator-valued function $P:\mathbb{R}_{+}\to L(X)$ is called coercive,
if it is positive and there exists $\mu>0,$ such that
$$\langle P(t)x,x\rangle\geq \mu\|x\|^{2}_{X} \quad \forall x\in X.$$
The following result guaranteeing the existence of a coercive solution for operator Lyapunov equations will be helpful in the sequel. It is related to analogous characterizations of exponential dichotomy given in \cite[Corollary 4.48]{carmen}.
\begin{prop}\label{hmfdrr}\cite[Theorem $5.2$]{damak2021input}.
Let $A : \R_+ \to L(X)$ be continuous and uniformly bounded in the uniform
operator topology.
Then, the evolution operator generated by $\{A(t)\}_{t\geq0}$ is uniformly exponentially stable if and only if there exists a continuously differentiable, bounded, coercive positive operator-valued function $P:\mathbb{R}_{+}\to L(X)$, satisfying the Lyapunov equality
\begin{equation}\label{R2}
A(t)^*P(t)+ P(t)A(t)+\dot{P}(t)=-I,\quad t\geq 0.
\end{equation}
\end{prop}
\begin{proof}
According to our assumption and arguments at the beginning of the
subsection, the family $\{A(t)\}_{t\geq0}$ generates a strongly continuous
evolution family of bounded operators $\{W(t,s)\}_{t\geq s\geq0}$. Assume
that $\{W(t,t_0)\}_{t\geq t_0\geq0}$ is uniformly exponentially stable. Define the operator-valued function $P$ by
$$P(t)=\int_{t}^{\infty}W(\tau,t)^{*}W(\tau,t)d\tau.$$
The integral converges by the uniform exponential stability.
 Thanks to the continuity of $t\mapsto A(t)$, the map $(\tau,t)\mapsto
 W(\tau,t)$ is uniformly continuous by \cite[Theorem 5.2, p.128]{Paz83},
 and thus $P$ is continuously differentiable. $P$ is coercive due to \cite[Lemma 5.1]{damak2021input}.
 Thus, we can compute
 $\dot{P}$ as in \cite[Theorem $5.2$]{damak2021input} and obtain
$$\dot{P}(t)= \int _{t} ^{\infty} W(\tau,t)\frac{\partial}{\partial
  t}W(\tau,t)d\tau+\int_{t}^{\infty }\frac{\partial}{\partial
  t}W(\tau,t)^*W(\tau,t)d\tau-I = -P(t)A(t) - A(t)^*P(t) -I. $$

For the converse direction see \cite[Theorem 5.2]{damak2021input}.
\end{proof}

Next, we give criteria for ISS of time-varying linear systems (\ref{unb}) with a family of the linear uniformly bounded operators 
$\{A(t)\}_{t\geq0}.$

\begin{thm} Let $A : \R_+ \to L(X)$ be continuous and uniformly bounded in the uniform
operator topology with associated evolution operator $\{W(t,t_0)\}_{t\geq t_0 \geq 0}$. 
The following statements are equivalent:
\begin{itemize}
\item[$(i)$] (\ref{unb}) is ISS. 
\item[$(ii)$] (\ref{unb}) is 0-UGAS.
\item[$(iii)$] (\ref{unb}) is iISS.
\item[$(iv)$] The evolution operator $\{W(t,t_0)\}_{t\geq t_0 \geq 0}$ is uniformly asymptotically stable.
\item[$(v)$] The evolution operator $\{W(t,t_0)\}_{t\geq t_0 \geq 0}$ is uniformly exponentially stable.
\item[$(vi)$] There exists a continuously
differentiable, bounded, positive, coercive operator-valued function
$P:\R_+ \to {\cal L}(X)$ satisfying \eqref{R2}, and the function  $V: \mathbb{R}_{+}\times X \to \mathbb{R}_{+}$ defined by
\begin{equation}\label{R1Ev}
V(t,x)=\langle P(t)x,x\rangle, \quad t\geq0, \quad x\in X,
\end{equation}
is an ISS Lyapunov function for \eqref{unb}.
\end{itemize}
\end{thm}

\begin{proof}
The equivalences (i)--(v) are covered by \cite[Theorem 4.1]{damak2021input},
\cite[Corollary 4.1]{damak2021input}, and \cite[Lemma
4.1]{damak2021input}.  It
follows from Theorem \ref{hjMMMMMljg} that (vi) implies (i).

(v) $\Longrightarrow$ (vi). As the evolution operator
$\{W(t,t_0)\}_{t\geq t_0 \geq 0}$ is uniformly exponentially stable,
Proposition \ref{hmfdrr} yields that there exists a continuously
differentiable, bounded, positive, coercive operator-valued function $P(\cdot)$ satisfying \eqref{R2}. Consider $V:\mathbb{R}_{+}\times X \to \mathbb{R}_{+}$ as defined in (\ref{R1Ev}). As $P$ is continuous, $V$ is continuous as well. Since $P$ is bounded and coercive, for some $\mu_1,p> 0$, it holds that
$$\mu_1\|x\|^{2}_{X}\leq V(t,x)\leq p\|x\|^{2}_{X}\quad \forall x \in X\quad \forall t\geq 0,$$
where $p=\displaystyle \sup_{t\in \mathbb{R}_{+}}\|P(t)\|$. Thus, (\ref{R1E}) holds.

For any given $x_0\in X,$ for any $u\in\mathcal{U},$ any $t_0\geq0,$ and
since the corresponding mild solution $x(t)=\phi(t,t_0,x_0,u)$ is a
piecewise classical solution by Proposition~\ref{prop:mild-is-classical},
it is piecewise continuously differentiable. Hence, with the exception of
countably many points the Lie derivative of $V$ along the trajectories of
system (\ref{unb}) satisfies (omitting the argument $t$ for legibility)
\begin{align*}
\dot{V}_{u}(t,x)&=\langle\dot{P}(t)x,x\rangle+\langle P(t)\dot{x},x\rangle +\langle P(t)x,\dot{x}\rangle\\
&=\langle \dot{P}(t)x,x\rangle+ \langle P(t)[A(t)x+B(t)u],x(t)\rangle+\langle P(t)x, A(t)x+B(t)u\rangle\\
&=\langle \dot{P}(t)x,x\rangle+\langle P(t)A(t)x,x\rangle+\langle
P(t)B(t)u,x\rangle +\langle P(t)x,A(t)x\rangle+\langle P(t)x,B(t)u\rangle.
\end{align*}
Since $\langle P(t)x,A(t)x\rangle=\langle A(t)^*P(t)x,x\rangle$, 
by applying the Lyapunov equation (\ref{R2}) 
and using the Cauchy-Schwarz inequality, we obtain
\begin{equation*}
    \dot{V}_{u}(t,x)\leq -\|x\|^{2}_{X}+2
    p\|x\|_{X}\|B(t)\|\|u\|_{\mathcal{U}} \leq -\|x\|^{2}_{X}+2
    p\|x\|_{X}\|B\|_\infty\|u\|_{\mathcal{U}},
\end{equation*}
where $\|B\|_{\infty}=\sup_{t\in \mathbb{R_{+}}}\|B(t)\|<\infty$. 
Utilizing Young's inequality we have for any $\eta>0$
$$\dot{V}_{u}(t,x)\leq -\|x\|^{2}_{X}+ p\|B\|_{\infty}\eta\|x\|^{2}_{X}+\frac{p\|B\|_{\infty}}{\eta}\|u\|^{2}_{\mathcal{U}},$$
Choosing
$\eta<(p\|B\|_{\infty})^{-1}$ this shows that $V$ is an ISS Lyapunov function for (\ref{unb}).
\end{proof}

\subsection{Evolution families}

In the remainder, all vector spaces will assumed to be Banach spaces.
In this subsection, we extend the analysis of ISS Lyapunov functions to
systems where the operators in the family $\{A(t): D(A(t))\subset X\to
X\}_{t\geq 0}$ are generally unbounded. 
As before, we assume that $\mathcal{U}=PC(\mathbb{R}_+, U)$. 
Again, we obtain that uniform exponential stability of the evolution family
$\{W(t,\tau)\}_{t\geq \tau\geq0}$, characterizes ISS 
and allows for the construction of noncoercive Lyapunov functions for systems with unbounded operators $A(\cdot)$.

In this subsection, we establish the equivalence between various stability concepts for the system (\ref{unb}) including ISS, 0-UGAS, iISS, and the uniform exponential stability of the evolution operator. Theorem \ref{hdkmnmlo} also proposes a non-coercive ISS Lyapunov function for the system.
\

We start with a simple lemma. 
\begin{lem}\label{hnlkfrankm}
Let $B\in C(\mathbb{R}_{+},L(X))$ with $\sup_{t\in \mathbb{R_{+}}}\|B(t)\| < \infty,$ and $\{W(t,s)\}_{t\geq s\geq0}$ be a strongly continuous evolution family. Then, for any $u\in\mathcal{U}$ and any $t\geq0$ it holds that 
$$ \displaystyle\lim _{h \to 0^{+}}\frac{1}{h}\int_{t}^{t+h}W(t,s)B(s)u(s)ds=B(t)u(t). $$
\end{lem}

%
%
%
%

Next, we generalize our results about ISS Lyapunov functions to the case of the unbounded operator $A(\cdot)$ and derive a constructive converse Lyapunov theorem for (\ref{unb}) with  bounded input operators $\{B(t)\}_{t\geq0}$. 
Our construction \eqref{eq:non-coercive ISS-LF-linsys} is motivated by a corresponding construction for linear time-invariant systems, see \cite{mironchenko2018lyapunov}.

\begin{thm}\label{hdkmnmlo}
Let the family $\{A(t)\}_{t\geq0}$ generate a strongly
continuous evolution family $\{W(t,s)\}_{t\geq s\geq0}.$ Assume
$B\in C(\mathbb{R}_{+},L(U,X)),$ with $\|B\|_{\infty}=\sup_{t\in \mathbb{R_{+}}}\|B(t)\|<\infty$. The following statements are equivalent for the system (\ref{unb}):
\begin{enumerate}
\item[(i)] (\ref{unb}) is ISS.
\item[(ii)] (\ref{unb}) is 0-UGAS.
\item[(iii)] (\ref{unb}) is iISS.
\item[(iv)] The evolution operator $\{W(t,t_0)\}_{t\geq t_0 \geq 0}$ is uniformly asymptotically stable.
\item[(v)] The evolution operator $\{W(t,t_0)\}_{t\geq t_0 \geq 0}$ is uniformly exponentially stable.
\item[(vi)] The function $V:\mathbb{R}_{+} \times X \rightarrow \mathbb{R}_{+},$ defined by
\begin{eqnarray}\label{rrrrhnlkfn}
V(t,x)=\int_{t}^{\infty}\|W(\tau,t)x\|_{X}^{2}d\tau,
\label{eq:non-coercive ISS-LF-linsys}
\end{eqnarray}
is \emph{a non-coercive ISS Lyapunov function} for (\ref{unb}) which is
locally Lipschitz continuous. 
\end{enumerate}
If any of the above equivalent properties hold, then for all $t\geq0,$ all $x\in X,$ all $u\in\mathcal{U},$ and all $\eta>0$ it holds that 
\begin{equation}\label{rhmnhlh}
\dot{V}_{u}(t,x)\leq-\|x\|_{X}^{2}+\frac{\eta k^{2}}{2w}\|x\|_{X}^{2}+\frac{k^{2}}{2\eta w}\|B\|_{\infty}\|u(t)\|_{U}^{2},
\end{equation}
where $k,w>0$ are obtained from (v) and are chosen such that
\begin{equation}\label{rhklet}
\|W(t,t_0)\|\leq k e^{-w(t-t_0)} \quad \forall t_0\geq0 \quad t\geq t_0.
\end{equation} 
\end{thm}
\begin{proof}
(i) $\Longleftrightarrow$ (ii) $\Longrightarrow$ (iii). This follows from
\cite[Theorem 4.1]{damak2021input}. From that result we use for that
ISS implies the so-called $L^1$-ISS property to conclude (iii) from (i), and (iii) $\Longrightarrow$ (ii) is  evident.
(ii) $\Longrightarrow$ (iv) follows directly from the definition of
0-UGAS. The equivalence of (iv) and (v) is contained in
Lemma~\ref{lemmmcc1}. That (vi) implies (v) follows from \cite[Theorem 1]{datko1972uniform}.

(v) $\Longrightarrow$ (vi). 
Let the evolution operator $\{W(t,t_0)\}_{t\geq t_0 \geq 0}$ be uniformly exponentially stable, i.e, there exist $k,w>0$ such that (\ref{rhklet}) holds.
Consider $V:\mathbb{R}_{+} \times X \rightarrow \mathbb{R}_{+},$ as defined in (\ref{eq:non-coercive ISS-LF-linsys}). 
We have for all $t\geq0$ and all $x\in X$ 
\begin{align*}
V(t,x)&\leq\int_{t}^{\infty}\|W(\tau,t)\|^{2}\|x\|^{2}_{X}d\tau \leq\frac{k^{2}}{2w}\|x\|^{2}_{X}.
\end{align*}
If $V(t,x)=0$, then by strong continuity, $\|W(\tau,t)x\|_{X}=0$ for all
$\tau\geq t$ and thus $x=0$ and hence the noncoercive condition \eqref{R1EN} holds.  
Next fix $x\in X,$ $t\geq0$, $u\in\mathcal{U}$, and estimate the
Lie derivative of $V$ as  
\begin{align*}
\dot{V}_{u}(t,x)&=\displaystyle\limsup _{h \to 0^{+}}\frac{1}{h}\big{(}V(t+h,\phi(t+h,t,x,u))-V(t,x)\big{)}\\
&=\displaystyle\limsup _{h \to 0^{+}}\frac{1}{h}\bigg{(}\int_{t+h}^{\infty}\|W(\tau,t+h)\phi(t+h,t,x,u)\|_{X}^{2}d\tau-\int_{t}^{\infty}\|W(\tau,t)x\|_{X}^{2}d\tau\bigg{)}.
\end{align*}
Using \eqref{bvbbLL4} we have for $\tau \geq t+h$
\begin{align*}
W(\tau,t+h)\phi(t+h,t,x,u) 
&= W(\tau,t+h)\Big(W(t+h,t)x_0+\int_{t}^{t+h}W(t+h,s)B(s)u(s)ds\Big)\\
&= W(\tau,t)x_0+\int_{t}^{t+h}W(\tau,s)B(s)u(s)ds.
\end{align*}
We obtain
\begin{align*}
\dot{V}_{u}(t,x)&=\displaystyle\limsup _{h \to 0^{+}}\frac{1}{h}\bigg{(}\int_{t+h}^{\infty}\Big\|W(\tau,t)x+\int_{t}^{t+h}W(\tau,s)B(s)u(s)ds\Big\|_{X}^{2}d\tau- \int_{t}^{\infty}\|W(\tau,t)x\|_{X}^{2}d\tau\bigg{)}\\
&\leq\displaystyle\limsup _{h \to 0^{+}}\frac{1}{h}\left(\int_{t+h}^{\infty}\|W(\tau,t)x\|_{X}^{2}-\int_{t}^{\infty}\|W(\tau,t)x\|_{X}^{2}d\tau\right) \\
&\ \ \ \ \ \ \ \ + \limsup _{h \to 0^{+}}\left(\int_{t+h}^{\infty}\Big\|\int_{t}^{t+h}W(\tau,s)B(s)u(s)ds\Big\|_{X}^{2}+2\|W(\tau,t)x\|_{X}\Big\|\int_{t}^{t+h}W(\tau,s)B(s)u(s)ds\Big\|_X d\tau\right)\\
&= J_{1}+J_{2},
\end{align*}
where
 \[
J_1=\displaystyle\limsup _{h \to 0^{+}}\frac{1}{h}\Big(\int_{t+h}^{\infty}\|W(\tau,t)x\|_{X}^{2}-\int_{t}^{\infty}\|W(\tau,t)x\|_{X}^{2}d\tau\Big),
\]
 and
$$J_2=\displaystyle\limsup _{h \to 0^{+}}\frac{1}{h}\int_{t+h}^{\infty}\bigg{(}\Big\|\int_{t}^{t+h}W(\tau,s)B(s)u(s)ds\Big\|_{X}^{2} + 2\|W(\tau,t)x\|_{X}\Big\|\int_{t}^{t+h}W(\tau,s)B(s)u(s)ds\Big\|_{X} \bigg{)}d\tau.$$
We have by continuity of $\tau \mapsto \|W(\tau,t)x\|_{X}^{2}$ that
\begin{equation*}
    J_1=\displaystyle\lim_{h \to 0^{+}}\frac{1}{h}\bigg{(}-\int_{t}^{t+h}\|W(\tau,t)x\|_{X}^{2}d\tau\bigg{)}=-\|x\|_{X}^{2}.
\end{equation*}

Now we proceed to $J_2:$
\begin{align*}
J_2&\leq\displaystyle\limsup _{h \to 0^{+}}\int_{t}^{\infty}\frac{1}{h}\Big\|\int_{t}^{t+h}W(\tau,s)B(s)u(s)ds\Big\|_{X}^{2}d\tau+\displaystyle\limsup _{h \to 0^{+}}\int_{t}^{\infty}2\|W(\tau,t)x\|_{X}\Big\|\frac{1}{h}\int_{t}^{t+h}W(\tau,s)B(s)u(s)ds\Big\|_{X}d\tau.
\end{align*}

The limit of the first summand equals zero since
\begin{align*}
\lim_{h \to
  0^{+}}\frac{1}{h}\int_{t+h}^{\infty}\Big\|\int_{t}^{t+h}W(\tau,s)B(s)u(s)ds\Big\|_{X}^{2}d\tau
&\leq k^{2} \|u\|^2_{\mathcal{U}}\|B\|^2_{\infty} \lim_{h \to 0^{+}} \int_{t+h}^{\infty}\frac{1}{h} 
\left(\int_{t}^{t+h} e^{-w (\tau-s)} ds \right)^2
d\tau\\
&= k^{2} \|u\|^2_{\mathcal{U}}\|B\|^2_{\infty} \lim_{h \to 0^{+}}
 \frac{\left( e^{wh}-1\right)^2 }{hw^2} \int_{t+h}^{\infty} e^{-2w(\tau-t)}
   d\tau = 0.  
\end{align*}
To bound the limit of the second summand, note that for each fixed $\tau$
\begin{align*}
2\|W(\tau,t)x\|_{X}\left\|\frac{1}{h}\int_{t}^{t+h}W(\tau,s)B(s)u(s)ds\right\|_{X}
&\leq 2k^{2}e^{-2w(\tau-t)}\|x\|_{X}\|B\|_{\infty}\|u\|_{\mathcal{U}}
\frac{1}{wh}\left(e^{wh}-1\right) .
\end{align*}
Applying the dominated convergence theorem and Lemma \ref{hnlkfrankm}, we obtain 
\begin{align}\label{jhgfyjk}
J_2&=\displaystyle\limsup _{h \to 0^{+}}\int_{t}^{\infty}2\|W(\tau,t)x\|_{X}\Big\|W(\tau,t)\int_{t}^{t+h}\frac{1}{h}W(t,s)B(s)u(s)ds\Big\|_{X}d\tau \notag\\
&=\int_{t}^{\infty}2\|W(\tau,t)x\|_{X}\|W(\tau,t)B(t)u(t)\|_{X}d\tau.
\end{align}
By using Young's inequality we obtain for any $\eta> 0$ that
\begin{equation}\label{nai23mr}
J_2\leq\int_{t}^{\infty}\eta\|W(\tau,t)x\|_{X}^{2}+\frac{1}{\eta}\|W(\tau,t)B(t)u(t)\|_{X}^{2}d\tau.
\end{equation}
Then by (\ref{rhklet}), we have
\begin{equation*}
J_2\leq\frac{\eta k^{2}}{2w}\|x\|^{2}_{X}+\frac{k^{2}}{2\eta
  w}\|B(t)\|^{2}\|u(t)\|^{2}_{U}
\leq\frac{\eta k^{2}}{2w}\|x\|^{2}_{X}+\frac{k^{2}}{2\eta w}\|B\|^{2}_{\infty}\|u(t)\|^{2}_{U}.
\end{equation*}
Therefore, 
$$\dot{V}_{u}(t,x)\leq -\|x\|_{X}^{2} + \frac{\eta k^{2}}{2w}\|x\|^{2}_{X} + \frac{k^{2}}{2\eta w}\|B\|^{2}_{\infty}\|u(t)\|^{2}_{U}.$$
Overall, for all $x\in X,$ all $u\in\mathcal{U},$ all $t\geq0,$ and all $\eta>0$
we obtain that the inequality (\ref{rhmnhlh}) holds. Choosing $\eta <\frac{2w}{k^{2}}$, this shows that $V$ is a non-coercive ISS Lyapunov function in dissipative form for (\ref{unb}). 

For the local Lipschitz continuity of $V,$ pick any $r>0,$ any $x,y\in X$ with $\|x\|_{X},\|y\|_{X}<r$ and any $t\geq0.$ We have: 
\begin{align*}
|V(t,x)-V(t,y)|&=\bigg{|}\int_{t}^{\infty}\|W(\tau,t)x\|_{X}^{2} - \|W(\tau,t)y\|_{X}^{2}d\tau\bigg{|}\\
&\leq\int_{t}^{\infty}\big{|}\|W(\tau,t)x\|_{X}^{2}-\|W(\tau,t)y\|_{X}^{2}\big{|}d\tau\\
&=\int_{t}^{\infty}\bigg{|}\big{(}\|W(\tau,t)x\|_{X}-\|W(\tau,t)y\|_{X}\big{)}\big{(}\|W(\tau,t)x\|_{X}+\|W(\tau,t)y\|_{X}\big{)}\bigg{|}d\tau.
\end{align*}
Since $\big{|}\|a\|_{X}-\|b\|_{X}\big{|}\leq\|a-b\|_{X},$ 
for all $a,b\in X$ and since $\|x\|_{X},\|y\|_{X}<r,$ one has by (\ref{rhklet}) that 
\begin{align*}
|V(t,x)-V(t,y)|&\leq\int_{t}^{\infty}2kr\|W(\tau,t)(x-y)\|_{X}d\tau\\
&\leq2kr\int_{t}^{\infty}ke^{-w(\tau-t)}\|x-y\|_{X}d\tau\\
&\leq\frac{2k^2r}{w}\|x-y\|_{X},
\end{align*}
which shows that $V$ is locally Lipschitz.
\end{proof}

\section{Lyapunov methods for local and integral ISS of semi-linear
  systems}
\label{sec:lyap-meth-local}

We reformulate the system (\ref{R1}) in the following form on a Banach space $X$:
\begin{equation}\label{psiunb}
\left\lbrace
\begin{array}{l}
\dot{x}(t)=A(t)x(t)+B(t)u(t)+\psi(t,x(t),u(t)),\qquad t\geq t_0 \geq 0,\\
x(t_0)=x_0,
\end{array}\right.
\end{equation}
where $\psi:\mathbb{R}_{+} \times X\times U\rightarrow X$ satisfies the
Assumption $(\mathcal{H}_1)$ and $B\in C(\mathbb{R}_{+},L(U,X)),$ with
$\displaystyle \sup_{t\geq0}\|B(t)\|<\infty.$ As in the case of system (\ref{R1}), we assume that inputs belong to the space $\mathcal{U}=PC(\mathbb{R}_+,U).$
In this section, we prove two Lyapunov results. In Theorem \ref{hndrhn23},
we construct a (non-)coercive LISS Lyapunov function for the system
(\ref{psiunb}). In Theorem \ref{hndkrnm}, under a certain condition on $\psi,$ we construct a (non-)coercive iISS Lyapunov function for the system (\ref{psiunb}). 
Since $B$ is continuous and $u$ is piecewise continuous, $\Psi(t,x,u)=B(t)u+\psi(t,x,u)$ satisfies the Assumption $(\mathcal{H}_1)$ and according to Proposition \ref{propooo1}, there is a unique maximal mild solution $\phi(\cdot,t_0,x_0,u)\in C([t_0,t_m],X)$ of system (\ref{unb}) for any data $(t_0,x_0,u),$
where $0<t_m=t_m(t_0,x_0,u)\leq\infty.$

\subsection{Constructions of LISS Lyapunov functions}
\label{subsec:constr-liss-lyap}

In this subsection, we provide a method for the construction of
non-coercive LISS Lyapunov functions for the system (\ref{psiunb}) under a
local linear boundedness assumption.

\begin{description}
\item[$(\mathcal{H}_2)$]  For each $a>0$, there exist $b>0$ and (sufficiently small) $\rho>0$ such that for all $x \in X,$ all $t\geq 0,$ all $u\in \mathcal{U}$ satisfying $\|x\|_{X},\|u\|_{\mathcal{U}}\leq\rho$ it holds that
\begin{equation}
\label{psiiii}
\|\psi(t,x,u)\|_{X} \leq a\|x\|_{X}+b\|u\|_{\Uc}.
\end{equation}
\end{description}

In the following theorem, we outline the conditions for constructing a non-coercive LISS Lyapunov function for the system (\ref{psiunb}).

\begin{thm}\label{hndrhn23}
Let the family $\{A(t)\}_{t\geq0}$ be the generator of a strongly
continuous evolution family $\{W(t,s)\}_{t\geq s\geq0}.$ Assume
$B\in C(\mathbb{R}_{+},L(U,X)),$ with $\|B\|_{\infty}=\sup_{t\in \mathbb{R_{+}}}\|B(t)\|<\infty.$
Consider system
(\ref{psiunb}) and assume that Assumptions $(\mathcal{H}_1)$ and $(\mathcal{H}_2)$ hold. 
If (\ref{unb}) is 0-UGAS, then the function $V:\mathbb{R}_{+} \times X \rightarrow \mathbb{R}_{+},$ defined as in (\ref{rrrrhnlkfn}) by
$$V(t,x)=\int_{t}^{\infty}\|W(\tau,t)x\|_{X}^{2}d\tau,$$
is a non-coercive LISS Lyapunov function for (\ref{psiunb}).
\end{thm}
\begin{proof}
%
%
%
%
Let (\ref{unb}) is 0-UGAS and pick $u\equiv0.$ According to Theorem
\ref{hdkmnmlo}, the evolution family $\{W(t,s)\}_{t\geq s\geq0}$ generated
by $\{A(t)\}_{t\geq0}$ is uniformly exponentially stable, that is there
exist $k,w>0$ such that $\|W(t,t_0)\|\leq ke^{-\omega(t-t_0)} $ holds for
all $t\geq t_0 \geq 0$.

Pick $a,b,\rho>0$ as in Assumption $(\mathcal{H}_2)$. For all $x \in X,$
all $u\in \Uc$ with $\|x\|_{X}, \|u\|_{\mathcal{U}}\leq\rho$, we apply a
similar analysis as in the proof of Theorem \ref{hdkmnmlo}, with the distinction that instead of $B(t)u(t)$, we consider $B(t)u(t) + \psi(t,x(t),u(t))$ in (\ref{jhgfyjk}).
Then, we obtain that 
\begin{align*}
\dot{V}_{u}(t,x)&\leq -\|x\|_{X}^{2}+\int_{t}^{\infty}2\|W(\tau,t)x\|\Big\|W(\tau,t)\Big(B(t)u(t)+\psi(t,x,u(t))\Big)\Big\|_{X}d\tau\\
&\leq -\|x\|_{X}^{2}+2\int_{t}^{\infty}\|W(\tau,t)\|^{2}\|x\|_{X}\|B(t)u(t)+\psi(t,x,u(t))\|_{X}d\tau\\
&\leq -\|x\|_{X}^{2}+\frac{k^{2}}{w}\|x\|_{X}\big{(}\|B(t)\|\|u(t)\|_{U}+\|\psi(t,x,u(t))\|_{X}\big{)}.
\end{align*}
 We continue the above estimates: 
\begin{align*}
\dot{V}_{u}(t,x)
&\leq -\|x\|_{X}^{2}+\frac{k^{2}}{w}\|x\|_{X}\bigg{(}\|B\|_{\infty}\|u(t)\|_{U}+a\|x\|_{X}+b\|u(t)\|_{U}\bigg{)}\\
&\leq\big{(} -1+\frac{k^{2}}{w}a\big{)}\|x\|_{X}^{2}+\frac{k^{2}}{w}\|x\|_{X}\big{(}\|B\|_{\infty}+b\big{)}\|u\|_{\mathcal{U}},
\end{align*}
where $\|B\|_{\infty}=\sup_{t\geq0}\|B(t)\|.$
Define $\kappa\in\mathcal{K} $ by $\kappa(r)=\sqrt{r},$ $r\geq0$. 
Whenever $\|u\|_{\mathcal{U}}\leq \kappa^{-1}(\|x\|_{X})=\|x\|^{2}_{X},$ we obtain
\begin{align*}
\dot{V}_{u}(t,x)
&\leq\big{(} -1+\frac{k^{2}}{w}a\big{)}\|x\|_{X}^{2}+\frac{k^{2}}{w}\big{(}\|B\|_{\infty}+b\big{)}\|x\|_{X}^{3}\\
&\leq\big{(} -1+\frac{k^{2}}{w}a\big{)}\|x\|_{X}^{2}+\frac{k^{2}}{w}\rho\big{(}\|B\|_{\infty}+b\big{)}\|x\|_{X}^{2}.
\end{align*}
Since $a$ and $\rho$ can be chosen arbitrarily small, the right-hand side can be estimated from above by some negative quadratic function of $\|x\|_{X}$.
According to Theorem \ref{hjMMMMMljg}, $V$ is a non-coercive LISS Lyapunov function for (\ref{psiunb}).


\end{proof}
\subsection{Constructions of iISS Lyapunov functions}
\label{subsec:constr-iiss-lyap}

Now, we turn our attention to time-varying bilinear systems of the form (\ref{psiunb}). 
In \cite{MiI16}, the equivalence between uniform global asymptotic stability and iISS was shown for bilinear time-invariant infinite-dimensional control systems. Additionally, in \cite{mironchenko2015note} a method for construction of non-coercive iISS Lyapunov functions was presented when the state is a Banach space. It was further extended in \cite{damak2021input} to bilinear time-varying infinite-dimensional control systems, where the operators in the family $\{A(t)\}_{t\geq0}$ are bounded. We will generalize these results to time-varying bilinear infinite-dimensional control systems where the operators in the family $\{A(t)\}_{t\geq0}$ are unbounded. Together with the results from \cite[Proposition 5]{mironchenko2015note}, we establish Lyapunov characterization of iISS for (\ref{psiunb}). For this purpose, we impose the following assumption
\begin{description}
\item[$(\mathcal{H}_3)$]  There exist $\gamma>0$ and $\delta\in \mathcal{K},$ so that for all $x\in X$, all $u\in \mathcal{U},$ and all $t\geq 0,$ we have
\begin{equation}\label{ibiISS}
\|\psi(t,x,u)\|_{X}\leq\gamma\|x\|_{X}\delta(\|u\|_{\mathcal{U}}).
\end{equation}
\end{description}

\begin{thm}\label{hndkrnm}
Let the family $\{A(t)\}_{t\geq0}$ be the generator of a strongly
continuous evolution family $\{W(t,s)\}_{t\geq s\geq0}.$ Assume
$B\in C(\mathbb{R}_{+},L(U,X)),$ with $\|B\|_{\infty}=\sup_{t\in \mathbb{R_{+}}}\|B(t)\|<\infty.$
Consider system
(\ref{psiunb}) and
let Assumptions $(\mathcal{H}_1)$ and $(\mathcal{H}_3)$ hold.
Assume that the evolution operator $\{W(t,s)\}_{t\geq s\geq0}$ is uniformly exponentially stable and 
let $V$ be defined as in (\ref{rrrrhnlkfn}). 
If $\{W(t,s)\}_{t\geq s\geq0}$ satisfies
\begin{equation}
\label{nmin23}
\|W(t,s)x\|_{X}\geq Me^{-\lambda(t-s)}\|x\|_{X},\qquad x \in X,\quad t \ge s\ge 0,
\end{equation}
then
\begin{equation}\label{rhnkdrnhp}
Z(t,x):=\ln(1+V(t,x)),\qquad x \in X,\quad t\geq0,
\end{equation}
is a coercive iISS Lyapunov function for (\ref{psiunb}). 
In particular,  (\ref{psiunb}) is iISS.
\end{thm}%
\begin{proof}
In view of \eqref{nmin23}, we have the following sandwich bounds:
\begin{align}
\label{eq:Sandwich-bounds-V}
\frac{2\gamma^{2}\lambda}{ M^{2}} \|x\|^{2}_{X} \leq V(t,x)&\leq\int_{t}^{\infty}\|W(\tau,t)\|^{2}\|x\|^{2}_{X}d\tau \leq\frac{k^{2}}{2w}\|x\|^{2}_{X},\qquad x \in X,\quad t\geq 0.
\end{align}

Defining $Z$ as in (\ref{rhnkdrnhp}) yields for any $u\in\mathcal{U}$ and any $t\geq 0$ the estimate:
\begin{eqnarray}\label{hkjkklhjj}
  \dot{Z}_{u}(t,x) &=& \frac{1}{1+V(t,x)} \dot{V}_{u}(t,x).
\end{eqnarray}
Using a similar analysis as in the proof of Theorem \ref{hdkmnmlo}, with the difference that instead of $B(t)u(t)$, we consider $B(t)u(t) + \psi(t,x(t),u(t))$ in (\ref{nai23mr}), we obtain the following estimate for any $\eta > 0$
$$
\dot{V}_{u}(t,x)\leq-\|x\|_{X}^{2}+\int_{t}^{\infty}\eta\|W(\tau,t)x\|_{X}^{2}d\tau+\frac{1}{\eta}\int_{t}^{\infty}\|W(\tau,t)\|_{X}^{2}\|B(t)u(t)+\psi(t,x,u(t)\|_{X}^{2}d\tau.
$$
Since the evolution family $\{W(t,s)\}_{t\geq s\geq0}$ generated by $\{A(t)\}_{t\geq0}$ is uniformly exponentially stable, there exist $k,w>0$ such that $\|W(t,t_0)\|\leq ke^{-\omega(t-t_0)} $ holds. It follows that
$$
\dot{V}_{u}(t,x)\leq\bigg{(}\frac{nk^{2}}{2w}-1\bigg{)}\|x\|_{X}^{2}+\frac{k^{2}}{2\eta w}\|B(t)u(t)+\psi(t,x,u(t)\|_{X}^{2}.
$$
Using the inequality $\|a+b\|_{X}^{2}\leq2\|a\|_{X}^{2}+2\|b\|_{X}^{2},$
for all $a, b\in X,$ we obtain
\begin{align*}
\dot{V}_{u}(t,x)&\leq\bigg{(}\frac{nk^{2}}{2w}-1\bigg{)}\|x\|_{X}^{2}+\frac{k^{2}}{\eta w} \big{(}\|B(t)\|^{2}\|u(t)\|^{2}_{U} + \|\psi(t,x,u(t)\|_{X}^{2} \big{)} \\
&\leq\bigg{(}\frac{nk^{2}}{2w}-1\bigg{)}\|x\|_{X}^{2}+\frac{k^{2}}{\eta w} \Big{(}\|B\|^{2}_{\infty}\|u(t)\|^{2}_{U}+{\gamma}^{2} \|x\|_{X}^{2}\delta^{2}(\|u(t)\|_{U})\Big{)}.
\end{align*}
Now (\ref{hkjkklhjj}) yields
\begin{equation}\label{hkdrharh32}
 \dot{Z}_{u}(t,x)= \frac{1}{1+V(t,x)}\bigg{(}\bigg{(}\frac{nk^{2}}{2w}-1\bigg{)}\|x\|_{X}^{2}+\frac{k^{2}}{\eta w} \Big{(}\|B\|^{2}_{\infty}\|u(t)\|^{2}_{U}+{\gamma}^{2} \|x\|_{X}^{2}\delta^{2}(\|u(t)\|_{U})\Big{)}\bigg{)}.
\end{equation}
Pick $\eta\in(0,\frac{2w}{k^{2}}),$ such that $\frac{nk^{2}}{2w}-1<0$ and due to (\ref{hkdrharh32}), we get
\begin{align*}
\dot{Z}_{u}(t,x)&\leq\bigg{(}\frac{nk^{2}}{2w}-1\bigg{)} \bigg{(} \frac{\|x\|_{X}^{2}}{1+\frac{k^{2}}{2 w}\|x\|_{X}^{2}}\bigg{)}+\frac{k^{2}}{\eta w}\|B\|^{2}_{\infty}\|u(t)\|^{2}_{U}+\gamma^{2}\bigg{(}\frac{\|x\|_{X}^{2}}{1+V(t,x)}\bigg{)}\delta^{2}(\|u(t)\|_{U})
\end{align*}
Using sandwich bounds \eqref{eq:Sandwich-bounds-V}, we have
\begin{align*}
\dot{Z}_{u}(t,x)&\leq\bigg{(}\frac{nk^{2}}{2w}-1\bigg{)} \bigg{(} \frac{\|x\|_{X}^{2}}{1+\frac{k^{2}}{2 w}\|x\|_{X}^{2}}\bigg{)}+\frac{k^{2}}{\eta w}\|B\|^{2}_{\infty}\|u(t)\|^{2}_{U}+\gamma^{2}\frac{\|x\|_{X}^{2}}{1+\frac{M^{2}}{2\lambda}\|x\|_{X}^{2}}\delta^{2}(\|u(t)\|_{U})\\
&\leq\bigg{(}\frac{nk^{2}}{2w}-1\bigg{)} \bigg{(} \frac{\|x\|_{X}^{2}}{1+\frac{k^{2}}{2 w}\|x\|_{X}^{2}}\bigg{)}+\frac{k^{2}}{\eta w}\|B\|^{2}_{\infty}\|u(t)\|^{2}_{U}+\frac{2\gamma^{2}\lambda}{ M^{2}}\delta^{2}(\|u(t)\|_{U}).
\end{align*}
This establishes the dissipativity inequality for $V$. In view of \eqref{eq:Sandwich-bounds-V}, the map $Z$ satisfies the sandwich bounds. 
Hence, $Z$ is a coercive iISS Lyapunov function for (\ref{psiunb}).
According to Theorem \ref{hjljg}, (\ref{psiunb}) is iISS.
\end{proof}

\section{Examples}
\label{sec:examples}

%
%
%
%

In this section, we will analyze two examples showing the applicability of our methods.

\subsection{Kuramoto-Sivashinskiy equation}


Let $X=L^{2}(0,1)$. Consider
the controlled time-varying Kuramoto-Sivashinsky (KS) equation
\begin{equation}\label{1EX}
\displaystyle\frac{\partial x(z,t)}{\partial t}=-\displaystyle\frac{\partial^{4} x(z,t)}{\partial z^{4}}-\varrho \displaystyle\frac{\partial^{2} x(z,t)}{\partial z^{2}}-\mu(t)x(z,t)+\frac{x(z,t)|\sin(t)|}{1+e^{-z t}x^{2}(z,t)}u(z,t),\quad z\in(0,1),\ t\geq t_0>0,\\
\end{equation}
where $t_0$ is the initial time, $\varrho > 0$ is known as the anti-diffusion parameter, $\mu:\mathbb{R}_+\to \mathbb{R}_+$ is a H\"older continuous function, and $u\in\Uc=PC( \mathbb{R}_+, \mathbb{R})$.

Consider (\ref{1EX}) with homogenous Dirichlet boundary conditions:
\begin{equation}\label{R1jjd}
x(0,t)=\displaystyle\frac{\partial x}{\partial z}(0,t)=x(1,t)=\displaystyle\frac{\partial x}{\partial z}(1,t)=0,\quad t\geq t_0>0.
\end{equation}
The system (\ref{1EX}) can be represented in the form $\dot{x}=Bx+\Psi(t,x,u),$ where 
$$
B:=-\displaystyle\frac{\partial^{4}}{\partial z^{4}}-\varrho \displaystyle\frac{\partial^{2}}{\partial z^{2}}, \quad \mbox{ with domain } \quad D(B)=H^{4}(0,1)\cap H_0^{2}(0,1),
$$
and
$$
\Psi(t,x(z,t),u(z,t)):= -\mu(t)x(z,t)+\frac{x(z,t)|\sin(t)|}{1+e^{-z t}x^{2}(z,t)}u(z,t).
$$

It is well known from \cite{cerpa2010null,cerpa2011local,liu2001stability} that $B$ is the infinitesimal generator of 
an analytic semigroup on $L^{2}(0,1)$ that we denote by $S(t)$. 

From the assumptions it follows that $\Psi$ satisfies the assumptions of
Proposition~\ref{propooo1}, and this proposition ensures that (\ref{1EX}) 
gives rise to a control system $\Sigma=(X,\mathcal{U},\phi),$ as in Definition~\ref{csyol} where $\phi$ is the mild solution of (\ref{1EX}). 

If in addition to $\mu$, also $u$ is H\"older continuous, then \cite[Theorem 3.3]{Hen81} (with $\alpha = 0$) ensures that the system (\ref{1EX}) has a unique maximal mild solution $\phi(\cdot,0,x_0,u)$ and this mild solution is even a 
\emph{classical solution} satisfying the integral equation (\ref{bvbbLL4}). 


\begin{prop}
\label{prop:KS-equation-iISS} 
If $\varrho <4\pi^{2},$ then (\ref{1EX}), (\ref{R1jjd}) is iISS.
\end{prop}

\begin{proof}
For a given $\varrho \in \mathbb{R},$ it is shown in \cite{liu2001stability}, that the operator $-B$ with the Dirichlet boundary conditions
(\ref{R1jjd}) has a countable sequence of eigenvalues $(\sigma_n)_{n\in \mathbb{N}},$ such that $\displaystyle\lim_{n\rightarrow\infty}\sigma_n=\infty.$
We define
$$
\sigma(\varrho )=\min_{n \in \N}\sigma_n(\varrho).
$$
Take
$$
Z(t,x)=(1+e^{-t}) \int_{0} ^{1 }x^{2}(z)dz,\quad x\in L^{2}(0,1),\ t\geq 0.
$$
Obviously,
$$
\|x\|_{L^{2}(0,1)}^{2}\leq Z(t,x)\leq 2\|x\|_{L^{2}(0,1)}^{2} \quad \forall (t,x)\in \mathbb{R}_{+}\times L^{2}(0,1).
$$

Take any $x \in X$ and any $u$ H\"older continuous. Since for such $u$ the solution $\phi$ is classical, the Lie derivative of $Z$ along the solution of the system (\ref{1EX}) is
\begin{eqnarray*}
\dot{Z}_{u}(t,x)&=&
-e^{-t}\|x\|_{L^{2}(0,1)}^{2}\\
   &&+2(1+e^{-t}) \int _{0}^{1 }x(z)\left(-\displaystyle\frac{\partial^{4} x(z)}{\partial z^{4}}-\varrho \displaystyle\frac{\partial^{2} x(z)}{\partial z^{2}}-\mu(t)x(z)+\frac{x(z)|\sin(t)|}{1+e^{-zt}x^{2}(z)}u(z,t)\right)dz.
\end{eqnarray*}
Since $1+e^{-z t}x^{2}(z)\geq 1,$ we have
\begin{eqnarray}\label{1nnnX}
\dot{Z}_{u}(t,x) &\leq&-2(1+e^{-t})\int _{0}^{1 }x(z)\frac{\partial^{4} x(z)}{\partial z^{4}}dz-2\varrho (1+e^{-t})\int _{0}^{1 }x(z) \displaystyle\frac{\partial^{2} x(z)}{\partial z^{2}}dz \nonumber\\
&&\ \ \ \ \ \ \ \ \ \ \ \ \ + 2(1+e^{-t})\int _{0}^{1 }x^{2}(z) u(z,t)dz.
\end{eqnarray}
Partial integration of the first and the second terms of (\ref{1nnnX}) together with the Dirichlet boundary condition (\ref{R1jjd}) leads to
\begin{equation}\label{R1jjjd}
\int _{0}^{1 }x(z)\frac{\partial^{4} x(z)}{\partial z^{4}}dz=\int _{0}^{1 }\left(\displaystyle\frac{\partial^{2} x(z)}{\partial z^{2}}\right)^{2}dz,
\end{equation}
and
\begin{equation}\label{R2jjjd}
 \int _{0}^{1 }x(z)\displaystyle\frac{\partial^{2} x(z)}{\partial z^{2}}dz=-\int _{0}^{1 }\left(\displaystyle\frac{\partial x(z)}{\partial z}\right)^{2}dz.
\end{equation}
Combining (\ref{R1jjjd}) and (\ref{R2jjjd}), we find
$$\dot{Z}_{u}(t,x)\leq -2(1+e^{-t})\left(\int _{0}^{1 }\left(\displaystyle\frac{\partial^{2} x(z)}{\partial z^{2}}\right)^{2}dz-\varrho \int _{0}^{1}\left(\displaystyle\frac{\partial x(z)}{\partial z}\right)^{2}dz\right)+2Z(t,x)\|u(\cdot,t)\|_{U}.$$
By Lemma 3.1 in \cite{liu2001stability}, we obtain
$$
\dot{Z}_{u}(t,x)\leq-2\sigma(\varrho )Z(t,x)+2Z(t,x)\|u(\cdot,t)\|_{U}.
$$
Let $\varrho <4\pi^{2}.$ Then, as stated in \cite[Lemma 2.1]{liu2001stability}, $\sigma(\varrho )>0.$
Consider the following coercive candidate iISS Lyapunov function:
$$
V(t,x):=\ln(1+Z(t,x)),\quad x \in X, \ t\geq 0.
$$
Then,
\begin{eqnarray*}
\dot{V}_{u}(t,x)&\le& \frac{-2\sigma(\varrho )Z(t,x)}{1+Z(t,x)}+2\frac{Z(t,x)}{1+Z(t,x)}\|u\|_{\mathcal{U}}\\
&\le&\frac{-2\sigma(\varrho )\|x\|_{L^{2}(0,1)}^{2}}{1+2\|x\|_{L^{2}(0,1)}^{2}}+2\|u\|_{\mathcal{U}}\\
&=&-\vartheta(\|x\|_{L^{2}(0,1)})+ \chi(\|u\|_{\mathcal{U}}),
\end{eqnarray*}
where $\vartheta(s)=\frac{\sigma(\varrho )s^{2}}{1+2s^{2}}$ and $\chi(s)=2s.$

Theorem \ref{hjljg} now shows iISS for H\"older continuous inputs. The claim for all piecewise continuous inputs follows by density arguments.
\end{proof}

\subsection{Controlled heat equation}
\label{sec:expl2}

Let $\nu>0,$ $\ell>0,$ $X=L^{2}(0,\ell)$ and $\Uc=PC( \mathbb{R}_+, \mathbb{R})$.
We consider the controlled heat equation
\begin{equation}\label{22EX}
\left\lbrace
\begin{array}{l}
\displaystyle\frac{\partial x(z,t)}{\partial t}=\nu\frac{\partial^{2} x(z,t)}{\partial z^{2}}+R(t)x(z,t)+\omega \sin(z) x(z,t)+u(z,t),\quad z\in(0,\ell),\quad t\geq t_0>0,\\
\\
\displaystyle x(0,t)=0=x(\ell,t),
\end{array}\right.
\end{equation}
where $\omega\in\mathbb{R},$ $t_0$ is the initial time, $u\in \mathcal{U}$, $\{R(t)\}_{t\geq 0}$ is a family of linear bounded operators on $X,$ satisfying $\sup_{t \in \mathbb{R}_{+} }\|R(t)\|=r<\infty,$ and such that $t \mapsto R(t)$ is H\"older continuous. 
\par The system (\ref{22EX}) can be written as $\dot{x}=Bx+\Psi(t,x,u)$, where 
$$
B:=\nu\displaystyle\frac{\partial^{2}}{\partial z^2}, \quad \mbox{with domain}\quad D(B)=H_{0}^{1}(0,\ell)\cap H^{2}(0,\ell).
$$ 

and
$$
\Psi: (t,x(z,t),u(z,t))\mapsto R(t)x(z,t)+ \omega\sin(z)x(z,t)+u(z,t).
$$
where $S(t)$ represents the analytic $C_0$-semigroup on $X$ generated by the operator $B.$ 

As the nonlinearity $\Psi(t,x,u)$ is continuous in $t$ and $u$ and locally Lipschitz continuous in $x,$ uniformly in $t$ and $u$ on bounded sets, Proposition \ref{propooo1} guarantees that the system (\ref{22EX}) is a well-posed control system satisfying the BIC property. 

If $u$ is in addition H\"older continuous, \cite[Theorem 3.3]{Hen81} ensures that the system (\ref{22EX}) has a unique classical solution for any given initial condition.

Consider the following coercive ISS Lyapunov function candidate:
$$V(t,x)=\|x\|_{L_{2}(0,\ell)}^{2}=\int_{0}^{\ell }x^{2}(z)dz.$$
For any $x\in X$ and any H\"older continuous $u\in\Uc$, the corresponding solution of (\ref{22EX}) is a classic solution, and thus for $(t,x)\in\mathbb{R}_{+}\times D(B),$ 
the Lie derivative of $V$ with respect to the system (\ref{22EX}) can be computed using the integration by parts:
\begin{eqnarray*}
\dot{V}_{u}(t,x)&=&
2\int _{0}^{\ell}x(z)\left(\nu\frac{\partial^{2} x(z)}{\partial z^{2}}+R(t)x(z)+\omega\sin(t z)x(z)+u(z,t)\right)dz\\
&\leq&-2\nu\int _{0}^{\ell}\left(\frac{\partial x(z)}{\partial z}\right)^{2}dz+(2r+\omega)\int _{0}^{\ell}x^{2}(z)dz+2\int _{0}^{\ell}x(z)u(z,t)dz.
\end{eqnarray*}
Utilizing Friedrich's inequality (see \cite{mitrinovic1991inequalities}) in the first term, we continue estimates:
$$\dot{V}_{u}(t,x)\leq-\frac{2\nu\pi^{2}}{\ell^{2}}V(t,x)+(2r+\omega)V(t,x)+2\int _{0}^{\ell}x(z)u(z,t)dz.$$
Using Young's inequality, we have for any $\epsilon>0,$
\begin{eqnarray*}
\dot{V}_{u}(t,x)&\leq&
\left(-\frac{2\nu\pi^{2}}{\ell^{2}}+2(r+\omega)+\epsilon\right)V(t,x)+\frac{1}{\epsilon}\int _{0}^{\ell}u^{2}(z,t)dz\\
&\leq&\left(-\frac{2\nu\pi^{2}}{\ell^{2}}+2(r+\omega)+\epsilon\right)V(t,x)+\frac{\ell}{\epsilon}\|u(\cdot,t)\|_{U}^{2}.
\end{eqnarray*}
To show that $V$ is an ISS Lyapunov function for the equation (\ref{22EX}), we assume that
$-\frac{2\nu\pi^{2}}{\ell^{2}}+2(r+\omega)+\epsilon<0.$ As $\epsilon>0$ can be chosen arbitrarily small, we obtain the following sufficient condition for ISS of the system (\ref{22EX}) for H\"older inputs:
$$r+\omega<\frac{\nu\pi^{2}}{\ell^{2}}.$$
The claim for all piecewise continuous inputs follows by a density argument.

\section{Conclusion}
\label{sec:conclusion}

In this paper, we have proved that the existence of an (L)ISS/iISS Lyapunov function implies the (L)ISS/iISS property of a general time-varying infinite-dimensional system. We investigated the input-to-state stability of time-varying evolution equations in terms of coercive and non-coercive ISS Lyapunov functions with piecewise-right continuous inputs. Indeed, we have seen in this paper some classes of time-varying linear systems with
bounded input operators in the construction of Lyapunov functions. In addition, new methods for the construction of non-coercive LISS/iISS Lyapunov functions for a certain class of time-varying semi-linear evolution equations with unbounded linear operator $\{A(t)\}_{t\geq t_{0}}$ have been proposed as well. The application to some specific time-varying partial differential equations shows the practical significance of the theoretical results. An interesting problem for future research is the development of Lyapunov methods for ISS and iISS of the time-varying nonlinear parabolic PDEs with boundary disturbances.

%
%
%
%
 %

\bibliographystyle{abbrv}
\bibliography{Mir_LitList_NoMir,MyPublications,fabian}

\section*{Appendix}

\begin{exm}\label{ex:Uniform-attractivity-and-uniform-stability}
We construct a uniformly attractive system for which the evolution operator $W(t,t_0)$ is not uniformly bounded. It suffices to consider a one-dimensional system. Define $A:\mathbb{R}_{+}\to\mathbb{R}$ by setting for $k=0,1,2,\ldots$
\begin{equation*}
A(t):=
\begin{cases}
-2\ln \left(2(k+1)^2\right), \quad & t\in [k, k+\frac{1}{2}),\\
\phantom {-} 2\ln(k+1),\quad & t\in [k+\frac{1}{2}, k+1)\,.
\end{cases}
\end{equation*}
With this, we have for all $k\geq 0$ the following observations:
\begin{equation*}
\sup_{s\in[0,1]}W\Big(k+1,k+s\Big)=W\Big(k+1,k+\frac{1}{2}\Big)=e^{2\ln(k+1)\frac{1}{2}}=k+1\,.
\end{equation*}
In particular, the associated system is not uniformly stable. On the other hand,
\begin{equation*}
W(k+1,k)=e^{2\ln (k+1)\frac{1}{2}}e^{-2\ln(2(k+1)^2)\frac{1}{2}}=(k+1)\cdot\frac{1}{2(k+1)^2}=\frac{1}{2(k+1)}\,,
\end{equation*}
and
\begin{equation*}
\sup_{s\in[0,1]}W(k+s,k)=W(k,k)=1.
\end{equation*}
We claim that the system is uniformly attractive. To this end, fix $\varepsilon >0$ and choose $m\in \mathbb{N}$ such that
\begin{equation*}
\frac{1}{2m}<\varepsilon.
\end{equation*}
We claim that $T(\varepsilon)=m+1$ is a suitable choice to satisfy the item 2 of Definition~\ref{def:StabAttr}.
Fix an arbitrary $t_0=k_0+\delta_0$ with $k_0\in\mathbb{N}\cup\{0\}$, and $0\leq\delta_0<1.$ Now pick any $t\ge t_0+m+1.$ Then, $t=k_0+M+s,$ where $M\ge m+1,$ and $s\in[0,1].$ Then,
\begin{align*}
W(t,t_0)= & W(t, k_0+M)W(k_0+M, k_0+M-1)
           \cdots
          W(k_0+2,k_0+1) W(k_0+1, t_0)
\end{align*}
and we obtain
\begin{equation*}
\left|W(t,t_0)\right|\leq 1 \cdot\frac{1}{2(k_0+M)}\cdots\frac{1}{2(k_0+2)}(k_0+1)<\frac{1}{2^m}<\varepsilon\,.
\end{equation*}
This proves uniform attractivity.
\end{exm}

\end{document}